%% file: main.tex
\documentclass[a4paper,11pt]{article}
\usepackage{preamble}
\title{Accelerated Proximal Gradient Method\\for Non-Strongly Convex Time-Varying Optimization}
\author{
  Tomoya Kamijima\thanks{Department of Mathematical Informatics, The University of Tokyo, Tokyo, Japan. \\Email: \texttt{kamijima-tomoya101@g.ecc.u-tokyo.ac.jp}.}
  \qquad
  Naoki Marumo\thanks{Department of Mathematical Informatics, The University of Tokyo, Tokyo, Japan. \\Email: \texttt{marumo@mist.i.u-tokyo.ac.jp}.}
  \qquad
  Akiko Takeda\thanks{Department of Mathematical Informatics, The University of Tokyo, Tokyo, Japan; RIKEN Center for Advanced Intelligence Project, Tokyo, Japan. Email: \texttt{takeda@mist.i.u-tokyo.ac.jp}.}
}
\date{}

\begin{document}
\maketitle
\vspace{-2em}

\begin{abstract}
    This paper studies time-varying composite optimization problems with convex but not necessarily strongly convex objective functions, on compact domains and with bounded temporal variability.
    As a baseline, we first analyze the time-varying proximal gradient (TV-PG) method.
    We show that its asymptotic function value gap is bounded by $O(\delta^{1/2})$ as $\delta\to0$, where $\delta$ measures the temporal variability of the objective function.
    This dependence on $\delta$ is tight in the worst case for TV-PG.
    Our main contribution is a time-varying accelerated proximal gradient method that improves the asymptotic upper bound on the function value gap to $O(\delta^{2/3})$.
    The method builds on a novel acceleration scheme whose momentum parameter is determined by the current iterates, with no explicit dependence on the iteration index.
    We establish the bounds for both the baseline and accelerated methods through a unified analysis based on potential functions.
    Numerical experiments on a synthetic problem and autoregression problems with real-world datasets demonstrate the tracking performance of the proposed algorithm.
\end{abstract}

\input{1_introduction.tex}
\input{2_problem.tex}
\input{3_tvapg.tex}
\input{5_numerical.tex}
\input{6_conclusion.tex}

\subsection*{AI use statement}
In this work, we used generative AI tools for the following tasks:
\begin{itemize}
  \item Helping develop theoretical models or conceptual frameworks, formulating mathematical claims, providing critical ingredients for proving mathematical claims, and assisting in proof writing, for the lower bound of TV-PG in \cref{thm:tvpgd_lower_bound} and the development of TV-APG building on TV-APGe.
  \item Generating synthetic data sets, designing experiments, cleaning and reformatting datasets, and implementing methods in \cref{sec:numerical}.
  \item Assisting with translation.
\end{itemize}
We have not used generative AI tools for proposing or refining hypotheses or interpreting results, and supporting qualitative or thematic data analysis is not applicable to this work.
Additionally, we used generative AI tools for the following tasks:
\begin{itemize}
  \item Suggesting experimental parameters and creating scientific figures in \cref{sec:numerical,sec:exp_details}.
  \item Sourcing/searching for information, identifying relevant literature, and summarizing existing literature.
  \item Editing a research paper to improve readability and formatting references.
\end{itemize}
We have reviewed all AI-assisted work as follows:
\begin{itemize}
  \item All LLM-assisted proofs were checked line by line by the authors.
  \item All LLM-generated code was verified and tested for correctness by an author.
  \item All LLM-collected information was checked through a manual literature survey by an author.
\end{itemize}
We take responsibility for the final content of this work, including text, claims, or artifacts produced with the aid of generative AI.

\subsection*{Acknowledgements}
This work was partially supported by JSPS KAKENHI (23K28041, 24K23853, and 25KJ1123) and JST CREST (JPMJCR24Q2).

\bibliographystyle{abbrvnat}
\bibliography{ref}

\clearpage
\appendix

\input{90_appendix_application.tex}
\input{90_appendix_problem.tex}
\input{91_appendix_tvpgd.tex}
\input{92_appendix_apg.tex}
\input{92_appendix_tvapg.tex}
\input{93_appendix_tvapge.tex}
\input{94_appendix_numerical.tex}

\end{document}

%% file: 1_introduction.tex
\section{INTRODUCTION}\label{sec:introduction}
Optimization problems in applications such as machine learning, control, and signal processing often involve objective functions that change over time. 
We consider the following problem:
\begin{align}
    \min_{x\in\R^d}F_k(x)\coloneqq f_k(x)+g(x),\label{eq:problem}
\end{align}
where $f_k:\R^d\to\R$ is a smooth convex function that changes over time $k\in\N$, and $g:\R^d\to\R\cup\{+\infty\}$ is a proper closed convex function that is time-invariant.
We also refer to each time step as a round.
When $g$ is the indicator function of a compact convex set $\mathcal X\subset\R^d$, i.e., $g(x)=0$ for all $x\in\mathcal X$ and $g(x)=+\infty$ otherwise, the problem \cref{eq:problem} reduces to the constrained problem 
\begin{align}
    \min_{x\in\mathcal X}f_k(x).\label{eq:problem_constrained}
\end{align}
Problem~\cref{eq:problem} has been studied mainly in the two contexts of time-varying convex optimization (TVCO)~\citep{popkov2005gradient} and online convex optimization (OCO)~\citep{zinkevich2003online}. 

In TVCO, the objective function $F_k$ changes gradually over time.
The goal is to ensure that the \emph{error} at each round tends to zero as the variation between $F_k$ and $F_{k+1}$ tends to zero.
The error is often measured by $\norm{x_k-x_k^*}$, where $x_k$ is the current solution and $x_k^*$ is a minimizer of $F_k$.
The main difficulty is that the limited computation time at each round permits only a few algorithmic iterations~\citep{dall2020optimization}.
Most studies assume strong convexity to establish theoretical guarantees, e.g., \citet{popkov2005gradient,simonetto2016class}, with a few exceptions discussed in \cref{sec:related_work}.

In OCO, the objective function $F_k$ can change abruptly and the main focus is the constrained problem~\cref{eq:problem_constrained}.
The goal is to guarantee that the (dynamic) regret is sublinear under suitable restrictions, that is, 
\begin{align}
    \sum_{k=1}^K (F_k(x_k)-F_k(u_k))
    =o(K)\label{eq:regret}
\end{align}
for a comparator sequence $(u_k)_{k=1}^K\subset\mathcal X$ whose choice determines the notion of regret.
Online gradient descent can achieve sublinear static regret, corresponding to the choice $u_k\equiv u$, without strong convexity~\citep{zinkevich2003online}.
However, sublinear regret does not ensure small $F_k(x_k)-F_k(u_k)$ at every round.

This work explores the intermediate setting between TVCO and OCO, where $F_k$ is not necessarily strongly convex but changes gradually over time.
By relaxing the strong convexity assumption imposed in TVCO to mere convexity, we obtain results that apply to a broader range of problems and allow greater modeling flexibility.
For example, regression with the pseudo-Huber loss over a moving data window yields convex objectives that are not necessarily strongly convex.
Even when each objective is strongly convex, its strong convexity parameter may become arbitrarily small as the data change (see \cref{sec:application_regression}).
By combining OCO with controlled temporal variation as in TVCO, we can obtain meaningful per-round guarantees.
\cref{tab:comparison} summarizes the differences between standard TVCO, OCO, and our setting.

\begin{table}[t]
    \centering
    \caption{Comparison of Problem Settings.}
    \setlength{\tabcolsep}{3pt}
    \begin{tabular}{c|c|c|c}
        \toprule
        & \textbf{TVCO} & \textbf{OCO} & \textbf{Ours} \\
        \midrule
        Small Per-Round Variation 
         & \checkmark & & \checkmark \\
        Severe Time Limitation & \checkmark & & \checkmark \\
        No Strong Convexity & & \checkmark & \checkmark \\
        Per-Round Guarantee & \checkmark & & \checkmark \\
        \bottomrule
    \end{tabular}
    \label{tab:comparison}
\end{table}

We begin by showing that, under mere convexity, the time-varying proximal gradient (TV-PG) method~\citep{simonetto2017time} has a tight $O(\delta^{1/2})$ asymptotic tracking barrier as $\delta\to0$, where $\delta$ measures the temporal variation of $F_k$.
The analysis of TV-PG reveals that the tracking guarantee derives from the convergence rate of the proximal gradient (PG) method.
This leads to the following natural question:

\medskip
\emph{Can we improve per-round tracking guarantees in non-strongly convex TVCO by incorporating an acceleration technique?}
\medskip

We answer this question affirmatively.
We develop a new acceleration method that improves this dependence to $O(\delta^{2/3})$.

\subsection{Contributions}
The main contributions are summarized as follows:
\begin{itemize}
    \item We bound TV-PG's function value gap from above and below without strong convexity.
    For the upper bound, we use a transparent proof strategy based on potential functions.
    \item We propose a time-varying accelerated PG (TV-APG) algorithm that can achieve a better performance guarantee than TV-PG.
    Building on the accelerated PG method described in \citet{d2021acceleration}, we develop a novel variant for TV-APG.
    \item We provide numerical experiments to demonstrate the tracking performance of TV-APG on a synthetic quartic problem and three autoregression problems. TV-APG performs best on the synthetic problem, while another variant can perform better on benchmark datasets, suggesting a trade-off between worst-case guarantees and empirical adaptivity.
\end{itemize}

\subsection{Related Work}\label{sec:related_work}
This subsection reviews related work on TVCO and acceleration methods.
Literature on OCO is reviewed in \cref{sec:related_work_oco}.

\paragraph{TVCO.}
\citet{popkov2005gradient} proposed iteratively applying gradient descent to TVCO and established a theoretical guarantee on its tracking performance under the strong convexity assumption.
\citet{simonetto2016class} introduced a prediction-correction framework to improve tracking performance.
Various prediction-correction methods have since been proposed~\citep{simonetto2017prediction,lin2019simplified,bastianello2019prediction,bastianello2023extrapolation,kamijima2025simple}.
Constrained problems and variational inequalities have also been studied under strong regularity~\citep{zavala2010real,dinh2012adjoint,dontchev2013euler,tang2022running}.
\citet{simonetto2020time} and \citet{dall2020optimization} survey TVCO and its applications.
In contrast, our work analyzes discrete-time TVCO without strong convexity.

Relaxing strong convexity to mere convexity in TVCO is nontrivial because most existing analyses rely on linear convergence of the optimization methods.
There are studies on time-varying optimization without strong convexity, but assumptions such as the Polyak--{\L}ojasiewicz condition~\citep{liu2018coordinate,kim2022online} or metric subregularity~\citep{bastianello2024robust} are imposed to ensure linear convergence. 
In contrast, \citet{simonetto2017time} proposed methods based on Krasnosel'ski\u{\i}--Mann iteration and provided theoretical guarantees without assuming conditions for linear convergence.
\citet{simonetto2017time} analyzed the averages over $k$ of $\norm{x_{k+1}-x_k}^2$ and $F_k(x_k)-F_k^*$, where $F_k^*$ is the optimal value of $F_k$, assuming a compact domain. 
\citet{madden2021bounds}, however, proposed a method for non-strongly convex TVCO and established a per-round guarantee by regularizing $F_k$, without assuming compactness of the domain (see \cref{sec:orpgd}). 
We obtain a per-round bound on $F_k(x_k)-F_k^*$ while retaining the compact-domain assumption and avoiding the regularization.
Avoiding regularization eliminates the need to tune a regularization parameter.
Our numerical experiments also show better performance than the regularized baseline.
We also note that, in the continuous-time limit, \citet{sun2020continuous} analyzed the static and dynamic regret of Nesterov's accelerated gradient method~\citep{nesterov1983method} for non-strongly convex TVCO.

\paragraph{Acceleration Methods.}
First-order acceleration dates back to Nesterov's accelerated gradient method, which attains the optimal $O(n^{-2})$ rate for smooth convex optimization~\citep{nesterov1983method}.
This principle was extended to PG methods for composite objectives~\citep{Tseng2008Accelerated,beck2009fast,Nesterov2013Gradient}.
In standard versions of these methods, the extrapolation parameters follow prescribed schedules based on the iteration number, rather than being determined by feedback from the iterates.

Methods based on Polyak step sizes instead choose their momentum parameters using information generated along the optimization trajectory, with no explicit dependence on the iteration index~\citep{barre2020complexity,goujaud2024quadratic}.
Related approaches in continuous time use closed-loop damping or time scaling determined by the evolving state, with no explicit dependence on time~\citep{attouch2023fast,attouch2025fast,maier2026near}.
These results do not establish an $O(n^{-2})$ objective-gap guarantee for PG with state-dependent extrapolation over general composite convex objectives.
Our method selects its extrapolation parameter from the current iterate pair, with no explicit dependence on the iteration index, and achieves an $O(n^{-2})$ objective-gap bound for composite convex objectives with a compact effective domain.

\subsection{Notation and Definitions}
We use the standard Landau notation $O(\cdot)$, $\Omega(\cdot)$, and $\Theta(\cdot)$ to denote asymptotic upper bounds, lower bounds, and matching upper and lower bounds, respectively.
The norm $\|\cdot\|$ denotes the Euclidean norm for vectors and the spectral norm for matrices.
A function $f\colon \R^d\to\R$ is \emph{$L$-smooth} on $\mathcal X$ if $\|\nabla f(x)-\nabla f(y)\|\leq L\|x-y\|$ for all $x,y\in \mathcal X$.
The set of all closed, convex, and proper functions from $\R^d$ to $\R\cup\{\infty\}$ is denoted by $\mathcal G$.
The proximal operator of a function $g\in\mathcal G$ is denoted by $\prox_g(x)\coloneqq\argmin_{y\in \R^d}g(y)+\frac{1}{2}\|y-x\|^2$.
For a function $F\colon\R^d\to\R\cup\{\infty\}$, let $\mathcal X_F^*\coloneqq\argmin_{x\in \R^d}F(x)$, $\Delta_F(x,y)\coloneqq F(x)-F(y)$, and $\Delta_F(x)\coloneqq \sup_{y\in \R^d} \Delta_F(x,y)$.

%% file: 2_problem.tex
\section{PROBLEM SETTING}\label{sec:problem}
We consider the TVCO problem~\cref{eq:problem}.
We compute $x_k$ using only $g,f_0,\ldots,f_{k-1}$ for consistency with the online setting, although some studies in TVCO allow access to $f_k$ when computing $x_k$.
Bounds for the iterate obtained after using $f_k$ can be derived by combining our bounds for $x_k$ with standard convergence results for time-invariant optimization.
The number of iterations per round must also be bounded because computation time is limited.

\subsection{Assumptions}
Let $\mathcal X\coloneqq\dom g$ and $\mathcal F_L$ be the set of all $L$-smooth convex functions on $\mathcal X$.
This paper assumes the following conditions.
\begin{assumption}\label{asm:problem}
    Let $D,L,\delta> 0$ be constants.
    \begin{enuminasm}
        \item $g\in\mathcal G$ and $\mathcal X$ is compact with $\norm{x-y}\leq D$ for all $x,y\in\mathcal X$.\label{asm:boundedness}
        \item $f_k\in\mathcal F_L$ for all $k$.\label{asm:function_class}
        \item $|\Delta_{f_{k+1}}(x,y)-\Delta_{f_k}(x,y)|\leq 2\delta$ for all $x,y\in\mathcal X$ and $k$.\label{asm:t_lipschitz}
    \end{enuminasm}
\end{assumption}

\cref{asm:boundedness} generalizes the assumption, common in OCO, that the feasible set is compact and convex~\citep{hazan2016introduction}.
\cref{asm:function_class} is standard in convex optimization.
\cref{asm:t_lipschitz} describes the temporal variation of the objective function and is specific to the time-varying setting.
\citet{simonetto2017time} analyzed the average of $\Delta_{F_k}(x_k)$ over $k$ under the same assumptions, except that \cref{asm:t_lipschitz} is replaced by the slightly stronger condition $|f_{k+1}(x)-f_k(x)|\leq\delta$.

Our interest is in slowly varying problems, where small $\delta$ makes per-round guarantees informative. Large values of $\delta$, by contrast, allow rapid changes in the objective function, which can make tracking impossible. Although OCO includes such problems, they are outside the scope of this paper. This small-$\delta$ regime is also natural in many applications. For example, suppose $f_k(x)=f(x;kh)$, where $f\colon\mathcal X\times\R\to\R$ and $h>0$ is the sampling period. If $f(x;t)$ is Lipschitz continuous in $t$ uniformly over $x\in\mathcal X$, then $\delta=O(h)$ as $h\to0$.

\subsection{Performance Measure}\label{sec:performance_measure}
Let $x_k^*\in\mathcal X_{F_k}^*$ be a minimizer of $F_k$.
Under strong convexity, the tracking error $\|x_k-x_k^*\|$ is commonly used to measure performance~\citep{popkov2005gradient,simonetto2016class}.
Without strong convexity, the minimizer is not necessarily unique.
Even when it is unique, the distance to the minimizer may be large despite a small function value gap (see \cref{sec:example_performance_measure}).
We therefore use the function value gap 
\begin{align}
    \Delta_k
    \coloneqq\Delta_{F_k}(x_k)
    = F_k(x_k)-\inf_{x\in\mathcal X}F_k(x)
\end{align}
as our performance measure.
In particular, any per-round upper bounds on $\Delta_k$ can be summed to obtain a bound on the dynamic regret.

%% file: 3_tvapg.tex
\section{ACCELERATED TRACKING}\label{sec:tvapg_unified}
Let $g\in\mathcal G$ and $f\in\mathcal F_L$.
Let $T_{f,g}\colon \R^d\to \mathcal X$ be the proximal gradient (PG) operator defined as 
\begin{align}
    T_{f,g}(x)\coloneqq\prox_{L^{-1}g}(x-L^{-1}\nabla f(x)).
\end{align}

\subsection{A Tight Baseline: TV-PG}\label{sec:tvpg}
If $f$ is strongly convex, the PG operator $T_{f,g}$ is a contraction mapping, which implies that the PG algorithm converges linearly.
This property plays a central role in the analysis of TVCO.
However, in the non-strongly convex setting, the PG operator is not necessarily a contraction mapping, and thus the analysis of TVCO becomes more challenging.

Let $F\coloneqq f+g$ and $p\geq0$.
The central idea of our analysis is to utilize the potential function $\Psi_{F,p}\colon\mathcal X^2\to\R$ defined as
\begin{align}
    \Psi_{F,p}(x,u)
    &\coloneqq p\Delta_F(x,u)+\frac L2\|x-u\|^2.
\end{align}
The following proposition shows that the PG operator $T_{f,g}$ does not increase the potential.
The proof is given in \cref{sec:pgd_potential_proof}.
\begin{proposition}[e.g., {\citet[Theorem 4.1]{li2025proximal}}]\label{prop:pgd_potential}
    Let $g\in\mathcal G$, $f\in\mathcal F_L$, $x,u\in\mathcal X$, and $p\geq0$.
    Then, we have $\Psi_{F,p+1}(T_{f,g}(x),u)\leq\Psi_{F,p}(x,u)$.
\end{proposition}

TV-PG~\citep{simonetto2017time} iteratively applies the PG operator to the time-varying function $F_k=f_k+g$ at each round $k$, i.e., $x_{k+1}=T_{f_k,g}^n(x_k)$, where $T_{f_k,g}^n$ is the $n$-fold composition of $T_{f_k,g}$.
In practical applications, the available computation time limits the number of PG steps $n$ per round.

We now give an upper bound on the function value gap $\Delta_k$ of TV-PG.
The potential function plays a central role in the proof.

\begin{theorem}\label{thm:tvpgd_new_error}
    Suppose that \cref{asm:problem} holds.
    Then, TV-PG guarantees
    \begin{align}
        \Delta_k\leq \frac{LD^2}{2nk}+D\sqrt{\frac{2L\delta}n}+2\delta
    \end{align}
    for all $k\geq1$.
\end{theorem}

\begin{proof}[Proof sketch]
    Fix $u\in\mathcal X$ arbitrarily.
    From \cref{prop:pgd_potential,asm:t_lipschitz}, we have
    \begin{align}
        \Psi_{k-m+i+1,i+1}
        &\leq\Psi_{k-m+i,i}+2n(i+1)\delta
    \end{align}
    for all $0\leq i<m\leq k$, where $\Psi_{k,i}\coloneqq\Psi_{F_k,ni}(x_k,u)$.
    Solving this recursion and using the definition of $\Psi_{k,i}$ yield
    \begin{align}
        \Delta_{F_k}(x_k,u)
        &\leq \frac{L\|x_{k-m}-u\|^2}{2nm}+(m+1)\delta\label{eq:tvpgd_new_error_proof_sketch}
    \end{align}
    for all $1\leq m\leq k$.
    Taking $u\in\mathcal X_{F_k}^*$ and $m\coloneqq\min\{k,\ceil{D\sqrt{L/(2n\delta)}}\}$ completes the proof.
    See \cref{sec:tvpgd_new_error_proof} for the full proof.
\end{proof}

The upper bound on TV-PG's function value gap $\Delta_k$ in \cref{thm:tvpgd_new_error} decreases at the rate $O(k^{-1})$ toward an asymptotic error bound.
Taking $k\to\infty$ gives a bound of $O(\sqrt{\delta/n})$ for small $\delta$.
Taking $\delta\to0$ also recovers the standard $O(n^{-1})$ convergence rate of PG for time-invariant convex problems~\citep{beck2017first}.
The dependence on $\sqrt{\delta/n}$ follows from \cref{eq:tvpgd_new_error_proof_sketch}.
The first term, of order $O(1/(nm))$, comes from the convergence rate of PG.
The second term, of order $O(m\delta)$, bounds the error caused by changes in the objective over the last $m$ rounds.
Balancing these two terms yields the $\sqrt{\delta/n}$ dependence.

\Citet{simonetto2017time} also analyzed TV-PG, but they bound the average of $\Delta_k$ over $k$ rather than $\Delta_k$ itself.
Their asymptotic bound includes the term $L\delta_\ast^2/2$.
Here, $\delta_\ast$ is an upper bound on $\|x_{k+1}^*-x_k^*\|$, where $x_k^*\in\mathcal X_{F_k}^*$.
Even when $\delta$ is small, the term $L\delta_\ast^2/2$ can be as large as $LD^2/2$, as illustrated in \cref{sec:example_performance_measure}.

\cref{thm:tvpgd_new_error} is tight in the sense that the asymptotic function value gap of TV-PG is lower bounded by $\Omega(D\sqrt{L\delta/n})$.
See \cref{sec:tvpgd_lower_bound_proof} for the proof.
\begin{theorem}\label{thm:tvpgd_lower_bound}
    For any $D,L>0$, $n\geq1$, and $0<\delta<LD^2/(16n)$, there exist $g\in\mathcal G$ and $(f_k)_{k\geq0}$ satisfying \cref{asm:problem} with these constants, such that $(x_k)_{k\geq0}$ generated by TV-PG satisfies
    \begin{align}
        \limsup_{k\to\infty}\Delta_k\geq\frac{D}{8}\sqrt{\frac{L\delta}{n}}.
    \end{align}
\end{theorem}

In particular, we can construct a sequence of functions $(f_k)_{k\geq0}$ such that the function value gap $\Delta_k$ of TV-PG diverges to infinity if we only assume $g\in\mathcal G$, \cref{asm:function_class}, and \cref{asm:t_lipschitz}.
See \cref{sec:necessity_example} for details.

The matching lower bound shows that the $O(\sqrt{\delta/n})$ dependence is intrinsic to TV-PG.
This raises the central question of this paper: can acceleration break this tracking barrier without strong convexity?

\subsection{State-Dependent Acceleration}\label{sec:apg}
Let $g\in\mathcal G$ satisfy \cref{asm:boundedness}, $f\in\mathcal F_L$, $\gamma_-\geq0$, and let $(x,x^{-})$ satisfy $x,\gamma_-x-(\gamma_--1)x^-\in\mathcal X$.
We start from the following update of an acceleration method:
\begin{align}
    \gamma
    &=\phi(\gamma_{-})\coloneqq\frac12\paren*{1+\sqrt{1+4\gamma_{-}^2}},\label{eq:apg1_equivalence_gamma}\\
    y
    &=x+\frac{\gamma_{-}-1}{\gamma}(x-x^{-}),\label{eq:apg1_equivalence_y}\\
    x^{+}
    &=T_{f,\tilde g_{x,\gamma}}(y),\label{eq:apg1_equivalence_x}
\end{align}
where 
\begin{align}
    \tilde g_{x,\gamma}(z)\coloneqq(1-\gamma^{-1})g(x)+\gamma^{-1}g(\gamma z-(\gamma-1)x).
\end{align}
The only difference from FISTA~\citep{beck2009fast} is that the update of $x^{+}$ uses $\tilde g_{x,\gamma}$ instead of $g$.
In fact, the updates \cref{eq:apg1_equivalence_gamma,eq:apg1_equivalence_y,eq:apg1_equivalence_x} are equivalent to \cref{alg:fista} (see \cref{sec:fista_d_update_equivalence_proof} for details).
We denote it by $\apg_1$.
$\apg_1$ keeps all iterates in $\mathcal X$.

\begin{algorithm}[t]
  \caption{$\apg_1(f,x,z,\gamma,n,L)$~\citep[Algorithm 20]{d2021acceleration}}
  \label{alg:fista}
  \begin{algorithmic}[1]
    \Require $f\in\mathcal F_L$, $x,z\in \mathcal X$, $\gamma\geq0$, $n\geq1$, $L>0$
    \For{$i=1,\ldots,n$}
        \State $\gamma\gets(1+\sqrt{1+4\gamma^2})/2$
        \State $y\gets(1-\gamma^{-1})x+\gamma^{-1}z$
        \State $z\gets\prox_{\gamma L^{-1}g}(z-\tfrac\gamma L\nabla f(y))$
        \State $x\gets(1-\gamma^{-1})x+\gamma^{-1}z$
    \EndFor
    \Ensure $x$, $z$, $\gamma$
  \end{algorithmic}
\end{algorithm}

We focus on the potential function as in \cref{sec:tvpg}.
Let $F\coloneqq f+g$ and $\gamma\geq0$.
The potential function $\tilde\Phi_{F,\gamma}\colon\mathcal X^3\to\R$ for $\apg_1$ is defined as
\begin{align}
    \tilde\Phi_{F,\gamma}(x,z,u)&\coloneqq\gamma^2\Delta_F(x,u)+\frac L2\norm{z-u}^2.\label{eq:potential_function_d}
\end{align}

The following proposition shows that \cref{alg:fista} does not increase the potential function.
\begin{proposition}[{\citet[Theorem 4.22]{d2021acceleration}}]\label{prop:fista_potential_decrease}
    Let $g\in\mathcal G$ satisfy \cref{asm:boundedness}, and let $f\in\mathcal F_L$.
    Fix $x,z,u\in\mathcal X$ and $\gamma\geq0$.
    Let $x^+,z^+\in\mathcal X$ be the output of \cref{alg:fista} initialized at $(x,z)$ with parameter $\gamma$.
    Then $\tilde\Phi_{F,\phi^n(\gamma)}(x^+,z^+,u)\leq\tilde\Phi_{F,\gamma}(x,z,u)$.
\end{proposition}

This proposition looks similar to \cref{prop:pgd_potential}, but there is a crucial difference.
For the sequence $(x^i)_{i\geq0}$ generated by PG, we fix $x=x^i$ for some $i\geq0$.
\cref{prop:pgd_potential} implies that $\Psi_{F,p+n}(T_{f,g}^n(x),u)\leq\Psi_{F,p}(x,u)$ for \emph{any} $p\geq0$ and $n\geq1$, thus taking $p=0$ yields
\begin{align}
    \Delta_F(T_{f,g}^n(x),u)
    \leq\frac{\Psi_{F,n}(T_{f,g}^n(x),u)}n
    &\leq\frac{LD^2}{2n},\label{eq:pg_explanation1}
\end{align}
where we used \cref{asm:boundedness}.
On the other hand, consider the sequence $(x^i,z^i)_{i\geq0}$ generated by $\apg_1$ and fix $(x,z)=(x^i,z^i)$ for some $i\geq0$.
In the setting of \cref{prop:fista_potential_decrease}, we have $\tilde\Phi_{F,\phi^n(\gamma)}(x^+,z^+,u)\leq\tilde\Phi_{F,\gamma}(x,z,u)$ for the \emph{specific} $\gamma$ used to update $(x^i,z^i)$.
Thus, we only have
\begin{align}
    \Delta_F(x^+,u)
    &\leq\frac{\gamma^2}{(\phi^n(\gamma))^2}\Delta_F(x,u)+\frac{LD^2}{2(\phi^n(\gamma))^2}\label{eq:apg1_explanation1}
\end{align}
for the specific $\gamma$, where we used \cref{asm:boundedness}.
Since $\phi^n(\gamma)\geq\gamma+\frac n2$, this guarantees an accelerated convergence rate of $O(n^{-2})$ for any fixed initial $\gamma$, but the coefficient of the $O(n^{-2})$ term depends on the initial $\gamma$.
This difference makes it difficult to guarantee the tracking performance of an $\apg_1$-based TVCO algorithm.

When trying to apply $\apg_1$ to TVCO, one may consider resetting $\gamma$ to $0$ at the beginning of each round, but this does not give a good guarantee. 
This is because it only yields a convergence rate for 1 round, while \cref{eq:tvpgd_new_error_proof_sketch} is derived from a convergence rate for $m$ consecutive rounds.

The above observation motivates us to design a variant of $\apg_1$ whose potential function does not depend on iteration-dependent parameters such as $\gamma$ in $\apg_1$.
To achieve this, we replace the momentum parameter $(\gamma_--1)/\gamma$ in \cref{eq:apg1_equivalence_y}, defined by the recursion \cref{eq:apg1_equivalence_gamma}, with a parameter $\beta\in(0,1]$ computed from the iterates.
Let $\gamma_n\coloneqq\phi^n(0)$ be the parameter $\gamma$ after $n$ iterations of \cref{eq:apg1_equivalence_gamma} starting from $\gamma_0=0$.
Since $\gamma+\tfrac12\leq\phi(\gamma)\leq\gamma+1$, we have $(\gamma_{n-1}-1)/\gamma_n=1-\Theta(n^{-1})$.
On the other hand, FISTA with a specific momentum parameter guarantees $\norm{x^n-x^{n-1}}=O(n^{-1})$~\citep[Corollary 3.2]{chambolle2015convergence}, where $x^n$ is the $n$-th iterate of FISTA.
Motivated by this bound, for the sequence $(x^n)_{n\geq0}$ generated by the algorithm we design, we assume that $\norm{x^n-x^{n-1}}=\Theta(D/n)$ in worst-case scenarios, where dependence on $D$ is for normalization.
Under this assumption, defining $\beta_n\coloneqq(1+\tfrac2D\norm{x^n-x^{n-1}})^{-1}$ yields $\beta_n=1-\Theta(n^{-1})$, which is consistent with the corresponding momentum coefficient in $\apg_1$.
This argument is heuristic and does not constitute a formal proof of convergence.

Based on the above heuristic arguments, our variant of $\apg_1$ updates $(x,x^-)$ with $x,2x-x^-\in\mathcal X$ to $(x^+,x)\in\mathcal X^2$ as follows:
\begin{align}
    \beta
    &=\paren{1+\tfrac2D\norm{x-x^-}}^{-1},\label{eq:apg2_equivalence_beta}\\
    y
    &=x+\beta(x-x^-),\label{eq:apg2_equivalence_y}\\
    x^+
    &=T_{f,\tilde g_{x,2}}(y).\label{eq:apg2_equivalence_x}
\end{align}
The parameter $\gamma$ of $\tilde g_{x,\gamma}$ in \cref{eq:apg1_equivalence_x} is replaced with $2$ in \cref{eq:apg2_equivalence_x} for technical reasons.
In particular, when $g$ is the indicator function of $\mathcal X$, then $\tilde g_{x,2}$ is the indicator function of $(\mathcal X+x)/2\subset\mathcal X$.

The updates \cref{eq:apg2_equivalence_beta,eq:apg2_equivalence_y,eq:apg2_equivalence_x} have an equivalent formulation given in the following proposition.
The proof is given in \cref{sec:apg2_equivalence_proof}.
\begin{proposition}\label{prop:apg2_equivalence}
    Let $g\in\mathcal G$ satisfy \cref{asm:boundedness} and let $x, z\in\mathcal X$, $x^{-}\coloneqq 2x-z$, and $f\in\mathcal F_L$.
    Then, the updates \cref{eq:apg2_equivalence_beta,eq:apg2_equivalence_y,eq:apg2_equivalence_x} and $z^+=2x^+-x$ are equivalent to
    \begin{align}
        \beta&=\paren{1+\tfrac2D\norm{x-z}}^{-1},\label{eq:apg2_beta}\\
        y&=(1-\beta)x+\beta z,\label{eq:apg2_y}\\
        z^{+}&=\prox_{2L^{-1}g}(2y-x-\tfrac2L\nabla f(y)),\label{eq:apg2_z}\\
        x^{+}&=\tfrac12(x+z^{+}).\label{eq:apg2_x}
    \end{align}
\end{proposition}

The resulting algorithm is summarized in \cref{alg:apg2} and we denote it by $\apg_2$.
Let $S_{f}\colon\mathcal X^2\to\mathcal X^2$ be the $\apg_2$ operator defined as 
\begin{align}
    S_{f}(x,z)
    &\coloneqq\apg_2(f,x,z,1,L,D),
\end{align}
that is, $S_{f}$ maps $(x,z)$ to $(x^+,z^+)$ according to \cref{eq:apg2_beta,eq:apg2_y,eq:apg2_z,eq:apg2_x}.

\begin{algorithm}[t]
  \caption{$\apg_2(f,x,z,n,L,D)$}
  \label{alg:apg2}
  \begin{algorithmic}[1]
    \Require $f\in\mathcal F_L$, $x,z\in \mathcal X$, $n\geq0$, $L,D>0$
    \For{$i=1,\ldots,n$}
        \State $\beta\gets(1+\tfrac2D\norm{x-z})^{-1}$
        \State $y\gets(1-\beta)x+\beta z$
        \State $z\gets\prox_{2L^{-1}g}(2y-x-\tfrac2L\nabla f(y))$
        \State $x\gets (x+z)/2$
    \EndFor
    \Ensure $x$, $z$
  \end{algorithmic}
\end{algorithm}

We conclude this section by presenting a potential function for $\apg_2$.
For $p\geq5$, let $\Phi_{F,p}\colon\mathcal X^3\to\R$ be the potential function defined as
\begin{align}
    &\Phi_{F,p}(x,z,u)\\
    &\coloneqq p^2\Delta_F(x,u)+\frac L2\norm{3(x-u)+(p-3)(z-x)}^2.\label{eq:potential_function}
\end{align}
The following proposition gives a recursion for the potential function under \cref{alg:apg2}.
The proof is given in \cref{sec:fista_new_potential_decrease_proof}.
\begin{proposition}\label{prop:fista_new_potential_decrease}
    Let $g\in\mathcal G$ be a function satisfying \cref{asm:boundedness}, $f\in\mathcal F_L$, $x,z,u\in\mathcal X$, and $p\geq5$.
    Then, we have 
    \begin{align}
        \Phi_{F,p+1}(S_{f}(x,z),u)
        &\leq\theta\Phi_{F,p}(x,z,u)+30(1-\theta)LD^2,
    \end{align}
    where $\theta\coloneqq(p+1)(p-2)p^{-2}\in(0,1)$.
\end{proposition}

The potential $\Phi_{F,p}$ and the recursion in \cref{prop:fista_new_potential_decrease} are independent of iteration-dependent algorithmic parameters.

\subsection{TV-APG}\label{sec:tvapg}
Time-varying APG (TV-APG) iteratively applies the $\apg_2$ operator to the time-varying objective function $f_k$ at each round $k$, i.e., $(x_{k+1},z_{k+1})=S_{f_k}^n(x_k,z_k)$.
The algorithm is summarized in \cref{alg:tvfista_new}.
The first step of TV-APG is a PG step, which allows us to bound the initial value of the potential function.
Indeed, taking $p=0$ in the potential recursion of PG in \cref{prop:pgd_potential} gives an upper bound on the potential function as in \cref{eq:pg_explanation1}, while the potential recursion of $\apg_2$ in \cref{prop:fista_new_potential_decrease} only holds for $p\geq5$.
See \cref{lem:uniform_bound} for details.

\begin{algorithm}[t]
  \caption{TV-APG}
  \label{alg:tvfista_new}
  \begin{algorithmic}[1]
    \Require $x_0\in\mathcal X$, $n\geq1$, $L,D>0$
    \State $x_0^+\gets z_0^+\gets\prox_{L^{-1}g}(x_0-L^{-1}\nabla f_0(x_0))$
    \State $(x_1,z_1)\gets\apg_2(f_0,x_0^+,z_0^+,n-1,L,D)$
    \For{$k=1,2,\ldots$}
        \State $(x_{k+1},z_{k+1})\gets\apg_2(f_k,x_k,z_k,n,L,D)$
    \EndFor
  \end{algorithmic}
\end{algorithm}

TV-APG can be viewed as an extrapolation-based prediction-correction method.
For comparison, consider the simplified prediction-correction algorithm of \citet{lin2019simplified} for unconstrained strongly convex problems.
This method also uses extrapolation and updates the solution according to
\begin{align}
    \hat x_{k+1}&=x_k+(x_k-x_{k-1}),\label{eq:spc_prediction}\\
    x_{k+1}&=(I-L^{-1}\nabla f_{k+1})^n(\hat x_{k+1}).
\end{align}
The simplified prediction-correction algorithm uses the prediction term $x_k-x_{k-1}$ in \cref{eq:spc_prediction}.
The TV-APG momentum term $\beta(x-x^{-})$ in the update \cref{eq:apg2_equivalence_y} can be viewed as a discounted version of this term.

The main result of this section is an upper bound on the function value gap of TV-APG.
The proof follows the strategy used for \cref{thm:tvpgd_new_error}, starting from the potential recursion established in \cref{prop:fista_new_potential_decrease}.
The proof also requires a uniform bound on the potential function, which we discuss in \cref{sec:uniform_potential_bound}.

\begin{theorem}\label{thm:tvfista_new_error}
    Suppose \cref{asm:problem} holds.
    Then, \cref{alg:tvfista_new} guarantees that
    \begin{align}
        \Delta_k
        &\leq\frac{30LD^2}{n^2k^2}+4\sqrt[3]{\frac{LD^2\delta^2}{n^2}}+5\delta\label{eq:tvfista_new_error}
    \end{align}
    for all $k\geq1$.
\end{theorem}

\begin{proof}[Proof sketch]
    Fix $u\in\mathcal X$ arbitrarily.
    From \cref{asm:t_lipschitz,prop:fista_new_potential_decrease}, we can derive
    \begin{align}
        \rho_{4+n(i+1)}\Phi_{k-m+i+1,i+1}
        &\leq\rho_{4+ni}\Phi_{k-m+i,i}+C_{L,D,\delta,n,i}
    \end{align}
    for all $1\leq i< m\leq k$, where $\rho_i\coloneqq(i-1)(i-2)/i$, $\Phi_{k,i}\coloneqq\Phi_{F_k,4+ni}(x_k,z_k,u)$, and $C_{L,D,\delta,n,i}$ is a constant depending on $L$, $D$, $\delta$, $n$, and $i$.
    Using this recursion, we can obtain
    \begin{align}
        \Delta_{F_k}(x_k,u)
        &\leq\frac{30LD^2}{n^2m^2}+\paren*{\frac m2+4}\delta\label{eq:tvfista_new_error_proof_sketch}
    \end{align}
    for all $1\leq m\leq k$, where we used the uniform upper bound on $\Phi_{k-m+1,1}$ described in \cref{sec:uniform_potential_bound}.
    Taking $u\in \mathcal X_{F_k}^*$ and $m\coloneqq\min\set{k,\ceil{4\sqrt[3]{LD^2/(n^2\delta)}}}$ completes the proof.
    See \cref{sec:tvfista_new_error_proof} for the full proof.
\end{proof}

The upper bound on TV-APG's function value gap $\Delta_k$ in \cref{thm:tvfista_new_error} decreases at the rate $O(k^{-2})$ toward an asymptotic error bound.
Taking $k\to\infty$ gives a bound of $O((\delta/n)^{2/3})$ for small $\delta$.
TV-APG therefore improves both the decay rate and the asymptotic error bound over TV-PG.
In the static convex setting, letting $\delta\to0$ shows that $\apg_2$ achieves the same $O(n^{-2})$ convergence rate as $\apg_1$~\citep[Corollary 4.23]{d2021acceleration}.

\begin{corollary}\label{cor:fista_convergence}
    Let $g\in\mathcal G$ be a function satisfying \cref{asm:boundedness}, $f\in\mathcal F_L$, and $x^0\in\mathcal X$.
    Let $x^n$ be the iterate obtained after $n-1$ iterations of \cref{alg:apg2} with initial point $x^1=z^1=T_{f,g}(x^0)$, i.e., $(x^n,z^n)=S_{f}^{n-1}(x^1,z^1)$.
    Then, $\Delta_F(x^n)\leq30LD^2/n^2$ for all $n\geq1$.
\end{corollary}

\begin{proof}
    Taking $k=1$ and $\delta\to0$ in \cref{thm:tvfista_new_error} gives the desired result.
\end{proof}

In \cref{eq:tvfista_new_error_proof_sketch}, the first term, of order $O(1/(nm)^2)$, comes from the convergence rate of $\apg_2$.
The second term, of order $O(m\delta)$, bounds the error caused by changes in the objective over the last $m$ rounds.
Balancing these terms yields the $O((\delta/n)^{2/3})$ term in \cref{eq:tvfista_new_error}.

\subsection{Estimate-Based Acceleration: TV-APGe}\label{sec:tvapge}
The asymptotic error of $O((\delta/n)^{2/3}+\delta)$ can also be achieved by applying $\apg_1$ (\cref{alg:fista}) with a specific choice of the momentum parameter.
However, it requires an additional assumption.

Consider the following update rule:
\begin{align}
    (x_{k+1},z_{k+1},\underline{\hspace{0.8em}})=\apg_1(f_k,x_k,z_k,\gamma_k,n,L).
\end{align}
At each round, we choose a new momentum parameter $\gamma_k$ so that it does not depend on the accumulated iteration count.
We therefore do not use the output $\phi^n(\gamma_k)$ as $\gamma_{k+1}$ in the next round.
Our analysis motivates the choice $\gamma_k=\sqrt[3]{nLD^2/\delta_k}$, where $\delta_k$ estimates the temporal variability $\delta$ at round $k$.
The estimates form a nondecreasing sequence and satisfy $0<\delta_k\leq\delta$ for all $k\geq0$.
Computing these estimates requires function evaluations, as described in \cref{sec:delta_estimation}.
We call this method TV-APG with estimate (TV-APGe).

TV-APGe guarantees the following upper bound on the asymptotic function value gap.
\begin{theorem}\label{thm:tvfista_final_error_short}
    Suppose \cref{asm:problem} holds.
    Then, TV-APGe guarantees that
    \begin{align}
        \limsup_{k\to\infty}\Delta_k
        &\leq\paren*{\frac12+2\sqrt[3]{\frac{\delta}{\delta_\infty}}}\sqrt[3]{\frac{LD^2\delta^2}{n^2}}+2\delta,\label{eq:tvfista_final_error_asymptotic_short}
    \end{align}
    where $\delta_\infty\coloneqq\lim_{k\to\infty}\delta_k$.
\end{theorem}

\cref{thm:tvfista_final_error_short} states that the asymptotic function value gap of TV-APGe is bounded by $O((\delta/n)^{2/3}+\delta)$ if $\delta_\infty$ is not too small compared to $\delta$.
When we have access to $\delta$, we can set $\delta_k$ to satisfy $\delta_\infty=\delta$, which results in $O((\delta/n)^{2/3}+\delta)$ asymptotic error.

TV-APGe achieves an asymptotic error bound of $O((\delta/n)^{2/3}+\delta)$ provided that the limit of the estimates $\delta_k$ is not too small relative to $\delta$.
It also uses function evaluations to compute these estimates.
\Cref{alg:tvfista_new}, in contrast, needs neither this additional assumption nor function evaluations. 
Nevertheless, TV-APGe can outperform \cref{alg:tvfista_new} in practice, as shown in \cref{sec:numerical}.

%% file: 5_numerical.tex
\section{NUMERICAL EXPERIMENTS}\label{sec:numerical}
We evaluate the methods in \cref{sec:tvapg_unified} on four TVCO problems: a synthetic quartic minimization problem and three autoregression problems using benchmark data. 
We compare TV-PG, TV-APG, and TV-APGe with the online regularized proximal gradient (ORPG) algorithm. 
ORPG is a natural extension of online regularized gradient descent~\citep{madden2021bounds}.
Further details on ORPG are provided in \cref{sec:orpgd}. 
We use $n=10$ for all methods to ensure the same number of gradient queries per round.
Experimental details and additional results are provided in \cref{sec:exp_details}.

\subsection{Quartic Function}\label{sec:quartic}
We consider the following TVCO problem:
\begin{align}
    \min_{x\in\mathcal X}f_k(x)\coloneqq \frac14\norm{x-y_k}_{A_k}^4,\quad\mathcal X\coloneqq[-1,1]^d,
\end{align}
where $\norm{x}_A\coloneqq\sqrt{x^\top Ax}$.
This is convex but not strongly convex.
See \cref{sec:quartic_details} for the details of the parameter settings.

TV-APG performed best, followed by TV-APGe, TV-PG, and ORPG (\cref{fig:quartic}).
As $h$ decreases, the terminal errors of TV-APG and TV-APGe decrease much faster than the theoretical bound suggests.
For example, when $h$ decreases from $0.1$ to $0.001$, the terminal error of TV-APG decreases by a factor of approximately $2\times10^{-5}$.
In comparison, the $h^{2/3}$ dependence of the asymptotic bound corresponds to a factor of $10^{-4/3}$.
These results suggest that the worst-case bound may be conservative for this problem.

\begin{figure}[tb]
    \centering
    \includegraphics[width=0.6\columnwidth]{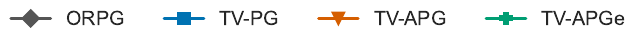}\\
    \includegraphics[width=0.5\columnwidth]{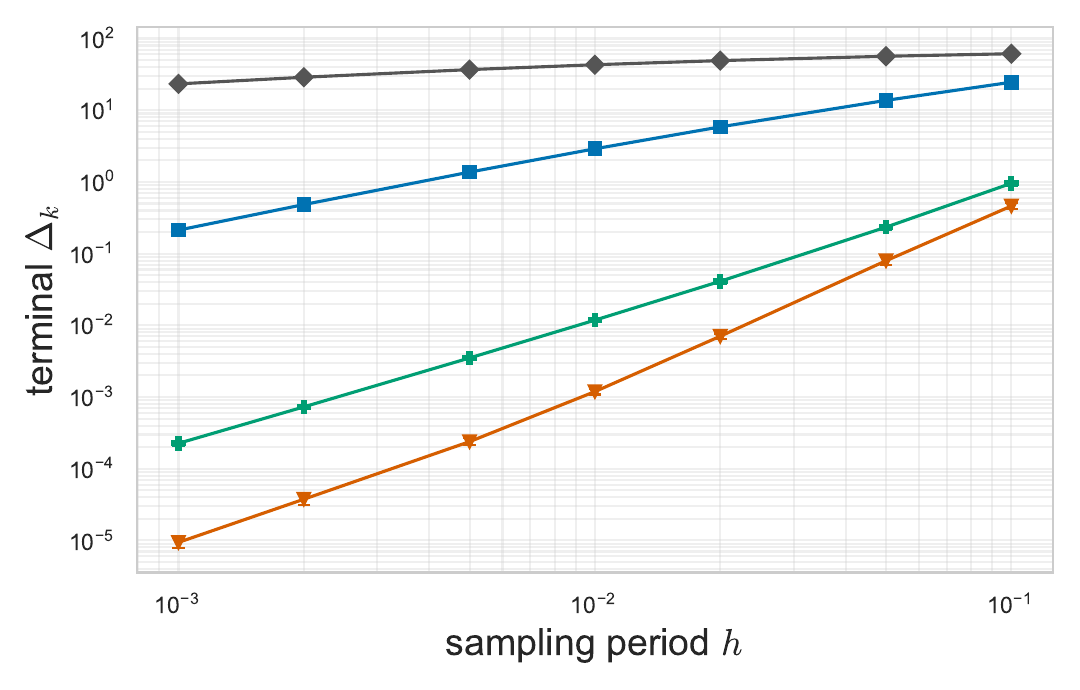}
    \caption{Quartic Function. The optimality gaps $\Delta_k$ at the final round are plotted for each sampling period $h$, where $\delta=O(h)$. The confidence intervals are shown as error bars, but they are too small to be visible.}
    \label{fig:quartic}
\end{figure}

\subsection{Time-Varying Autoregression}\label{sec:autoregression}
We use three time-series datasets from the Monash Time Series Forecasting Archive~\citep{godahewa2021monash}: Electricity Hourly (Electricity), Traffic Hourly (Traffic), and Solar 10 Minutes (Solar).
Solar contains simulated solar power data~\citep{godahewa2020solar10min}.

For each source dataset, let $\mathcal D$ denote the collection of preprocessed time series, and denote each series by $z=(z_k)_{k=0}^{17543}\subset\R$.
We consider the following time-varying autoregression problem:
\begin{align}
    \min_{x\in\mathcal X} f_k(x)
    &\coloneqq\frac1{|\mathcal D_k|}\sum_{(a,b)\in\mathcal D_k}\ell(\inner{a,x}-b),\label{eq:autoregression}\\
    \mathcal X&\coloneqq\Set{x\in\R^{1+168}}{\norm{x}\leq R},
\end{align}
for $k=887,888,\ldots,17543$, where 
\begin{align}
    \mathcal D_k
    &\coloneqq\{(a,b)\mid a=(1,z_{k'-1},\ldots,z_{k'-168}),~b=z_{k'},\\&\qquad\qquad\quad z\in\mathcal D,~k'\in[k-167,k]\}
\end{align}
and we use a pseudo-Huber loss function $\ell(r)\coloneqq\sqrt{1+r^2}-1$.
The strong convexity parameter depends on the data and may become arbitrarily small as the data change, so a common positive parameter over an indefinitely continuing data stream is not guaranteed. Consequently, tracking results that require such a parameter need not apply even when every objective is strongly convex.
See \cref{sec:autoregression_details} for details of the preprocessing of the datasets and the choice of parameters.

The results are shown in \cref{fig:electricity,fig:traffic,fig:solar}.
TV-APGe tended to perform best, followed by TV-APG, TV-PG, and ORPG.
Its advantage over TV-APG suggests a possible trade-off between worst-case performance and empirical adaptivity.
TV-APGe requires function evaluations, whereas the other methods do not.
TV-PG, TV-APG, and TV-APGe outperformed ORPG on all three datasets, possibly because our methods avoid the solution bias induced by adding a regularization term to the objective function.

\begin{figure*}[tb]
    \centering
    \includegraphics[width=0.6\columnwidth]{figures/legend_four_methods.pdf}\\
    \begin{minipage}[b]{0.33\textwidth}
        \centering
        \includegraphics[width=\columnwidth]{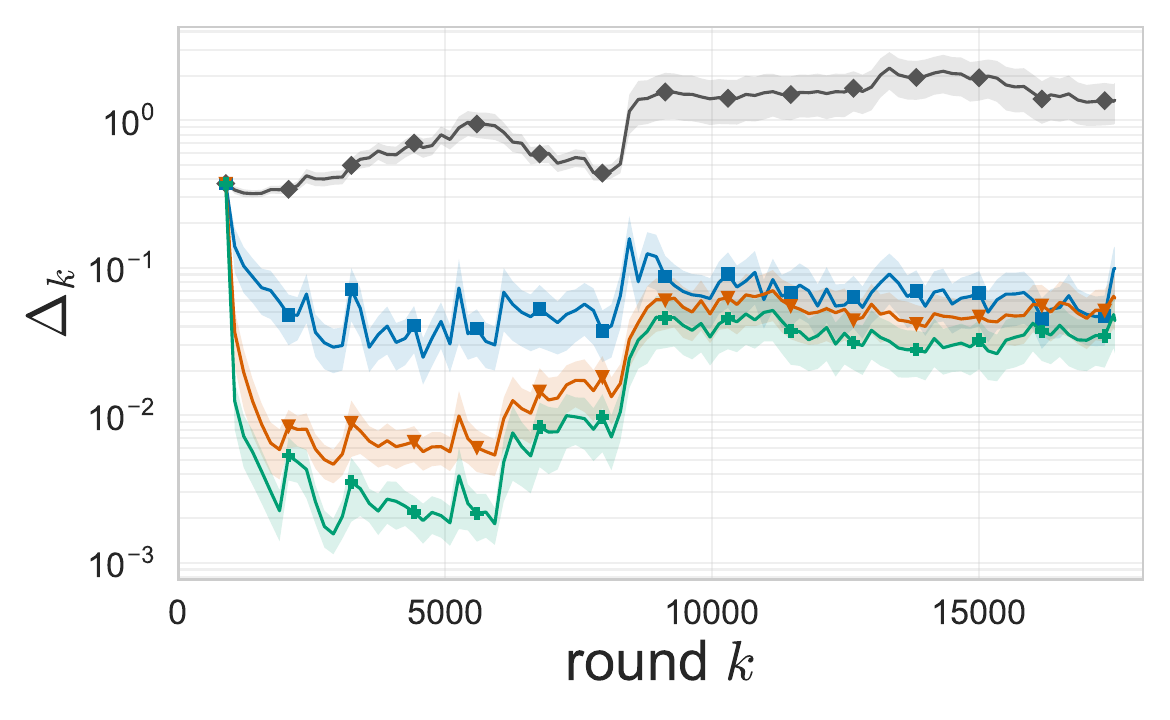}
        \subcaption{Electricity.}
        \label{fig:electricity}
    \end{minipage}%
    \begin{minipage}[b]{0.33\textwidth}
        \centering
        \includegraphics[width=\columnwidth]{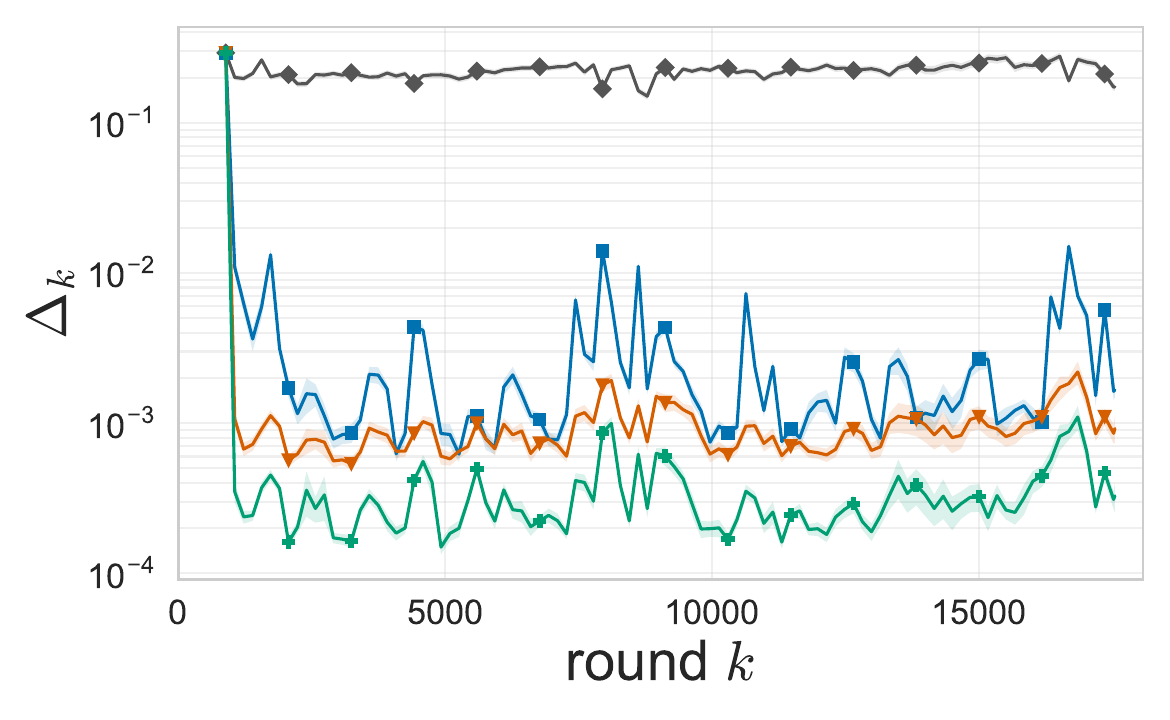}
        \subcaption{Traffic.}
        \label{fig:traffic}
    \end{minipage}%
    \begin{minipage}[b]{0.33\textwidth}
        \centering
        \includegraphics[width=\columnwidth]{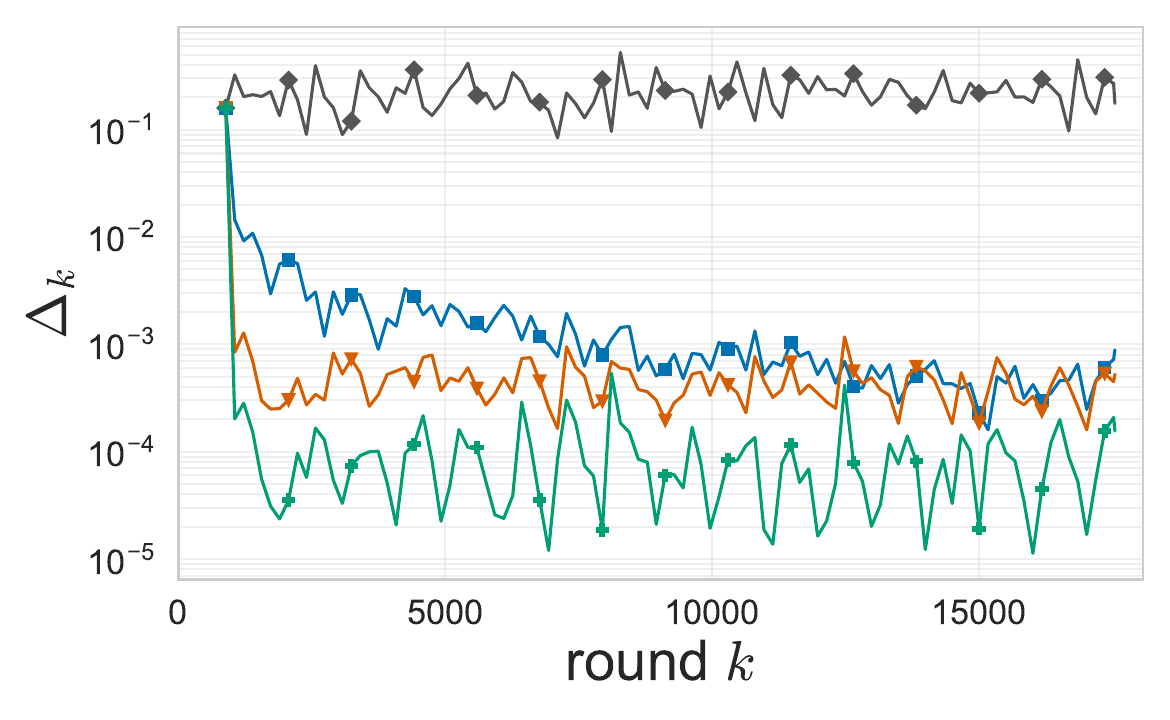}
        \subcaption{Solar.}
        \label{fig:solar}
    \end{minipage}
    \caption{Autoregression. The optimality gaps $\Delta_k$ are plotted every $168$ rounds and at the final round. The confidence intervals are shown as shaded bands.}
    \label{fig:autoregression}
\end{figure*}

%% file: 6_conclusion.tex
\section{CONCLUSION}\label{sec:conclusion}
We studied TVCO without strong convexity and proposed TV-APG.
The proposed method achieves an asymptotic error of $O((\delta/n)^{2/3}+\delta)$, which is smaller than that of TV-PG.
Numerical experiments demonstrated that TV-PG, TV-APG, and TV-APGe maintain small optimality gaps.

Future work includes analyzing the tracking performance of non-strongly convex TVCO methods under more sophisticated assumptions.
In particular, refining the temporal variation assumption is important for explaining why TV-APG performs better empirically than its bound suggests (\cref{sec:quartic}).
Another important direction is to develop an algorithm whose theoretical guarantees are comparable to those of TV-APG and whose empirical performance is comparable to that of TV-APGe.
It would also be useful to extend $\apg_2$ to unbounded time-invariant problems and compare its performance with that of existing acceleration methods.

%% file: 90_appendix_application.tex
\section{DETAILS FOR SECTION~\ref{sec:introduction}}\label{sec:application}
\subsection{Robust Regression with a Moving Data Window}\label{sec:application_regression}

In forecasting and system identification, a model can be fitted to a moving window of observations to adapt to changes in the data.
When observations arrive frequently, only a few optimization steps may be feasible before the next update~\citep{dall2020optimization}.
Robust losses reduce the influence of large residuals, while parameter constraints encode prior bounds on the model.

Let $(a_i,b_i)\in\R^d\times\R$ denote the observations, and let $N\geq1$ be the window size.
After collecting an initial window, consider
\begin{align}
    \min_{x\in\mathcal X}f_k(x)&\coloneqq\frac1N\sum_{i=k-N+1}^{k}\ell_\kappa(\inner{a_i,x}-b_i),\qquad
    \mathcal X\coloneqq\{x\in\R^d:\|x\|\leq R\},
\end{align}
where $\ell_\kappa$ is the pseudo-Huber loss defined as $\ell_\kappa(r)\coloneqq\kappa^2(\sqrt{1+(r/\kappa)^2}-1)$ and $\kappa,R>0$ are fixed.

For the pseudo-Huber loss,
\begin{align}
    \nabla^2 f_k(x)
    =\frac1N\sum_{i=k-N+1}^{k}
    \left(1+\frac{(\inner{a_i,x}-b_i)^2}{\kappa^2}\right)^{-3/2}a_ia_i^\top,
\end{align}
which implies that $f_k$ is not strongly convex in general.
Even if each $f_k$ is $\mu_k$-strongly convex, $\inf_{k\in\N}\mu_k$ may be zero.

Suppose that $\|a_i\|\leq A$ and $|b_i|\leq B$ uniformly, for constants $A,B>0$.
Since $0<\ell_\kappa''(r)\leq1$, we may take $L=A^2$ and $D=2R$.
Since $0\leq\ell_\kappa(\inner{a_i,x}-b_i)\leq \ell_\kappa(AR+B)$ on $\mathcal X$, we have
\begin{gather}
    |f_{k+1}(x)-f_k(x)|
    =\frac1N\abs{\ell_\kappa(\inner{a_{k+1},x}-b_{k+1})-\ell_\kappa(\inner{a_{k-N+1},x}-b_{k-N+1})}
    \leq\frac{\ell_\kappa(AR+B)}N,\\
    \bigl|\Delta_{f_{k+1}}(x,y)-\Delta_{f_k}(x,y)\bigr|
    \leq\frac{2\ell_\kappa(AR+B)}N
\end{gather}
for all $x,y\in\mathcal X$.
Hence, \cref{asm:problem} holds.
For multiple time series, the same argument applies by considering the average loss over the observations in each time slice.

\subsection{OCO and Dynamic Regret}\label{sec:related_work_oco}
\citet{zinkevich2003online} introduced the OCO framework and established a static regret bound for online gradient descent.
The same work also gave a dynamic regret bound for this method in terms of the path length defined by
\begin{align}
    P_K(u_1,\ldots,u_K)\coloneqq\sum_{k=1}^{K-1}\|u_k-u_{k+1}\|.
\end{align}
Subsequent work has studied OCO algorithms and dynamic regret bounds using measures of cumulative variation in the objectives or comparator sequence, such as path length~\citep{yang2016tracking,zhang2018adaptive,cutkosky2020parameter}, function variation~\citep{besbes2015non,zhang2018dynamic,wan2023improved}
\begin{align}
    V_K\coloneqq\sum_{k=1}^{K-1}\sup_{x\in\mathcal X}|f_k(x)-f_{k+1}(x)|,
\end{align}
and deviation from the dynamical model~\citep{hall2013dynamical}.
Other studies use multiple measures~\citep{jadbabaie2015online,zhao2020dynamic,zhao2024adaptivity}.

Dynamic regret has also been analyzed under the strong convexity assumption~\citep{mokhtari2016online,zhang2017improved,zhao2021improved}.
In this paper, we assume gradual temporal variation and derive guarantees on $F_k(x_k)-F_k^*$ at each round, while dynamic regret is the sum of $F_k(x_k)-F_k^*$ over $k$.
We focus on the regime in which measures of cumulative variation, such as $P_K$ and $V_K$, grow linearly with $K$.
Regret bounds in this regime are not sublinear.

In particular, \citet{wan2023improved} proved a dynamic regret bound of $O(\sqrt{K(1+V_K)})$ for online gradient descent.
If we assume $|f_{k+1}(x)-f_k(x)|\leq\delta$, then $V_K\leq(K-1)\delta$ and the bound becomes $O(\sqrt{K(1+K\delta)})$.
This yields
\begin{align}
    \frac1K\sum_{k=1}^K\Delta_k
    \leq O(K^{-1/2}+\delta^{1/2}),
\end{align}
while \cref{thm:tvpgd_new_error} states that $\Delta_k\leq O(k^{-1}+\delta^{1/2})$ for each $k$.

It is also worth noting that some existing methods require $O(\log K)$ updates or exact minimization at each round~\citep{jadbabaie2015online,zhang2018dynamic,cutkosky2020parameter,zhao2024adaptivity}, which is not suitable for our setting where computation time is limited at each round.

\subsection{ORGD and ORPG}\label{sec:orpgd}
\citet{madden2021bounds} proposed online regularized gradient descent (ORGD) for smooth TVCO problems without strong convexity.
This is the case $g=0$ in \cref{eq:problem}.
For $\mu>0$ and $\bar x\in\R^d$, define the regularized objective $f_k(x;\mu)\coloneqq f_k(x)+\tfrac\mu2\|x-\bar x\|^2$.
At each round, ORGD applies a gradient descent step to this objective with step size $2/(L+2\mu)$ instead of $L^{-1}$.

Let $x_k^*(\mu)\coloneqq\argmin_{x\in\R^d}f_k(x;\mu)$ and $\sigma(\mu)\coloneqq\sup_{k\geq1}\norm{x_k^*(\mu)-x_{k-1}^*(\mu)}$ for $\mu>0$.
Assume that $\mathcal X_{f_k}^*\neq\varnothing$ for all $k$.
Define
\begin{align}
    \sigma(0)\coloneqq
    \sup_{k\geq1}\|x_k^*(0)-x_{k-1}^*(0)\|,\quad\text{where}\quad
    x_k^*(0)\coloneqq\lim_{\mu\downarrow0}x_k^*(\mu).
\end{align}
ORGD has the following theoretical guarantee.
\begin{theorem}[Adapted from {\citet[Theorem 5.1]{madden2021bounds}}]\label{thm:orgd_asymptotic_error}
    Let $f_k$ be convex and $L$-smooth on $\R^d$ and $\mathcal X_{f_k}^*\neq\varnothing$ for all $k$.
    Suppose $\sigma(0)>0$ and the following bounded drift condition holds:
    \begin{align}
        \lim_{\mu\downarrow0}R(\mu)<\infty,\quad\text{where}\quad R(\mu)\coloneqq\sup_{k\geq0}\norm{x_k^*(\mu)-\bar x}.
    \end{align}
    Then, there exists $\mu\in(0,L]$ such that, given $x_0\in\R^d$, ORGD constructs a sequence $(x_k)_{k\geq0}$ with 
    \begin{align}
        \limsup_{k\to\infty}\norm{\nabla f_k(x_k)}
        &\leq2\sqrt2L\sqrt{\sigma(\mu)R(\mu)}.
    \end{align}
\end{theorem}

We add the condition $\sigma(0)>0$ to the theorem of \citet{madden2021bounds}.
Without this condition, the statement can fail even for a static problem.
To see this, let $f_k(x)=L(x-a)^2/2$ for all $k$, with $a\ne0$ and $\bar x=0$.
Then $x_k^*(\mu)=La/(L+\mu)\eqqcolon x^*(\mu)$, so $R(\mu)=L|a|/(L+\mu)$ is bounded and $\sigma(\mu)=0$.
For every $\mu>0$, ORGD converges to $x^*(\mu)$.
At this point, the gradient norm is $|\nabla f_k(x^*(\mu))|=L\mu|a|/(L+\mu)>0$, whereas the claimed upper bound is zero.

The additional condition $\sigma(0)>0$ excludes the counterexample and ensures the existence of a positive regularization parameter satisfying $\mu=L\sqrt{\sigma(\mu)/(2R(\mu))}>0$, which is essential for the proof of \cref{thm:orgd_asymptotic_error} in \citet{madden2021bounds}.

The bounded drift condition requires the unconstrained regularized minimizers to remain uniformly bounded over time and for all sufficiently small $\mu>0$, whereas \cref{asm:boundedness} requires the entire feasible domain $\mathcal X$ to be compact.
A further difference concerns the upper bound.
The bound on the asymptotic gradient norm in \cref{thm:orgd_asymptotic_error} is meaningful only when $\sigma(\mu)$ is small.
For small $\mu$, however, a small value of $\delta$ in \cref{asm:t_lipschitz} does not necessarily imply a small value of $\sigma(\mu)$.
This is because small changes in the objective function need not lead to small changes in its minimizer.

In our experiments, we generalize ORGD to the composite setting by applying the proximal gradient update to the regularized objective function $f_k(x)+g(x)+\tfrac\mu2\|x-\bar x\|^2$ with step size $2/(L+2\mu)$ at each round.
We call the resulting method online regularized proximal gradient (ORPG) and present it in \cref{alg:orpgd}.

The analysis of ORGD by \citet{madden2021bounds} motivates choosing $\mu$ to satisfy $\mu=L\sqrt{\sigma(\mu)/(2R(\mu))}$.
This relation is obtained by minimizing an upper bound while treating $\sigma(\mu)$ and $R(\mu)$ as fixed.
Finding a value of $\mu$ that satisfies this equation is difficult, so we use $\mu=L\sqrt{v/(2D)}$ as an approximation in our experiments.
Here, $v$ is an estimated upper bound on the distance between optimal solutions in consecutive rounds.
As in \citet{madden2021bounds}, we choose $\bar x$ to be the center of the feasible set, which is $0$ in all our experiments.

\begin{algorithm}[t]
  \caption{ORPG}
  \label{alg:orpgd}
  \begin{algorithmic}[1]
    \Require $x_0,\bar x\in \mathcal X$, $n\geq1$, $\mu,L>0$
    \For{$k=0,1,\ldots$}
        \State $x\gets x_k$
        \For{$i=1,\ldots,n$}
            \State $x\gets\prox_{(2/(L+2\mu))g}(x-\frac2{L+2\mu}(\nabla f_k(x)+\mu(x-\bar x)))$
        \EndFor
        \State $x_{k+1}\gets x$
    \EndFor
  \end{algorithmic}
\end{algorithm}

%% file: 90_appendix_problem.tex
\section{EXAMPLE ON THE PERFORMANCE MEASURE}\label{sec:example_performance_measure}
\begin{example}\label{ex:tracking_error_vs_function_value_gap}
    To illustrate the difference between the tracking error and the function value gap, consider the following example:
    \begin{align}
        \min_{x\in\mathcal X} f_k(x) \coloneqq \paren*{kh-\tau}x,\quad\mathcal X\coloneqq\bracket*{-\frac D2,\frac D2},
    \end{align}
    where $\tau,h>0$ are constants.
    This problem satisfies \cref{asm:problem} with $\delta=Dh/2$.
    Assume that $k_0<\tau/h<k_0+1$ for some $k_0$.
    The optimal solution is given by $x_k^*=D/2$ for $k\leq k_0$ and $x_k^*=-D/2$ for $k\geq k_0+1$.
    In this setting, even the exact minimization scheme $x_{k+1}\in\argmin_{x\in\mathcal X}f_k(x)$ can yield $\norm{x_k-x_k^*}=D$, no matter how small $h>0$ is.
    This scheme is idealized because computational time limitations make exact minimization impractical.
    Indeed, the exact minimization scheme outputs $x_k=D/2$ for $1\leq k\leq k_0+1$ and $x_k=-D/2$ for $k\geq k_0+2$.
    Thus $\|x_{k_0+1}-x_{k_0+1}^*\|=D$, which is independent of $h$.
    For the function value gap, we have $\Delta_{k_0+1}\leq Dh=2\delta$, which goes to zero as $h\to0$.
    As this example illustrates, $\norm{x_k-x_k^*}$ may not be a meaningful performance measure when the optimal solution jumps abruptly, while $\Delta_k$ is meaningful even in such a case.
\end{example}

%% file: 91_appendix_tvpgd.tex
\section{DETAILS FOR SECTION~\ref{sec:tvpg}}\label{sec:tvpgd_new_details}
\subsection{Proof of \cref{prop:pgd_potential}}\label{sec:pgd_potential_proof}
Before proving \cref{prop:pgd_potential}, we introduce a property of the PG operator.
\begin{proposition}[e.g., {\citet[Theorem 10.16]{beck2017first}}]\label{prop:pgd_property}
    Let $g\in\mathcal G$, $f\in\mathcal F_L$, and $F\coloneqq f+g$.
    Then, for all $x,y\in\mathcal X$, we have
    \begin{align}
        \Delta_F(T_{f,g}(x),y)
        &\leq \frac L2\norm{x-y}^2-\frac L2\norm{T_{f,g}(x)-y}^2,\label{eq:pgd_fejer}\\
        \Delta_F(T_{f,g}(x),x)
        &\leq0.\label{eq:pgd_descent}
    \end{align}
\end{proposition}

\begin{proof}
    By the definition of the proximal operator, we have
    \begin{align}
        T_{f,g}(x)
        &=\argmin_{z\in\R^d}\set*{g(z)+\frac L2\norm{z-(x-L^{-1}\nabla f(x))}^2}\\
        &=\argmin_{z\in\R^d}\set*{g(z)+f(x)+\inner{\nabla f(x),z-x}+\frac L2\norm{z-x}^2}
        =\argmin_{z\in\R^d}\tilde F_x(z),
    \end{align}
    where $\tilde F_x(z)\coloneqq g(z)+f(x)+\inner{\nabla f(x),z-x}+(L/2)\norm{z-x}^2$.
    Since $\tilde F_x$ is $L$-strongly convex, we have
    \begin{align}
        \tilde F_x(T_{f,g}(x))+\frac L2\norm{T_{f,g}(x)-y}^2
        &\leq \tilde F_x(y).
    \end{align}
    Because $f\in\mathcal F_L$, its smoothness gives $F(T_{f,g}(x))\leq \tilde F_x(T_{f,g}(x))$.
    Its convexity also gives
    \begin{align}
        \tilde F_x(y)
        &=g(y)+f(x)+\inner{\nabla f(x),y-x}+\frac L2\norm{y-x}^2\\
        &\leq F(y)+\frac L2\norm{y-x}^2.
    \end{align}
    Combining these inequalities yields \cref{eq:pgd_fejer}.
    Taking $y=x$ in \cref{eq:pgd_fejer} gives \cref{eq:pgd_descent}.
\end{proof}

Using this property, we can prove \cref{prop:pgd_potential} as follows.

\begin{proof}[Proof of \cref{prop:pgd_potential}]
    From \cref{prop:pgd_property} with $y=u$, we have
    \begin{align}
        \Psi_{F,p+1}(T_{f,g}(x),u)
        &= (p+1)\Delta_F(T_{f,g}(x),u)+\frac L2\|T_{f,g}(x)-u\|^2\\
        &=p\Delta_F(x,u)+p\Delta_F(T_{f,g}(x),x)+\paren*{\Delta_F(T_{f,g}(x),u)+\frac L2\|T_{f,g}(x)-u\|^2}\\
        &\leq p\Delta_F(x,u)+\frac L2\|x-u\|^2\\
        &= \Psi_{F,p}(x,u),
    \end{align}
    which completes the proof.
\end{proof}

\subsection{Proof of \cref{thm:tvpgd_new_error}}\label{sec:tvpgd_new_error_proof}
\begin{proof}[Proof of \cref{thm:tvpgd_new_error}]
    Fix $u\in\mathcal X$ arbitrarily.
    \cref{prop:pgd_potential} guarantees that
    \begin{align}
        \Psi_{F_k,p+n}(x_{k+1},u)
        =\Psi_{F_k,p+n}(T_{f_k,g}^{n}(x_k),u)
        &\leq\Psi_{F_k,p}(x_k,u).
    \end{align}
    From \cref{asm:t_lipschitz}, we have
    \begin{align}
        \Psi_{F_{k+1},p+n}(x_{k+1},u)
        &\leq\Psi_{F_k,p+n}(x_{k+1},u)+2(p+n)\delta.
    \end{align}
    Combining these two inequalities gives
    \begin{align}
        \Psi_{F_{k+1},p+n}(x_{k+1},u)
        &\leq\Psi_{F_k,p}(x_{k},u)+2(p+n)\delta
    \end{align}
    for all $k\geq0$ and $p\geq0$.
    In particular, we have
    \begin{align}
        \Psi_{k-m+i+1,i+1}
        &\leq\Psi_{k-m+i,i}+2n(i+1)\delta
    \end{align}
    for all $0\leq i<m\leq k$, where $\Psi_{k,i}\coloneqq\Psi_{F_k,ni}(x_k,u)$.
    Summing this over $i=0,\ldots,m-1$ yields
    \begin{align}
        \Psi_{k,m}
        &\leq\Psi_{k-m,0}+nm(m+1)\delta
        =\frac L2\|x_{k-m}-u\|^2+nm(m+1)\delta.
    \end{align}
    Since $\Psi_{k,m}\geq nm\Delta_{F_k}(x_k,u)$, we have
    \begin{align}
        \Delta_{F_k}(x_k,u)
        \leq\frac L{2nm}\|x_{k-m}-u\|^2+(m+1)\delta
        &\leq \frac{LD^2}{2nm}+(m+1)\delta
    \end{align}
    for all $1\leq m\leq k$, where the last inequality follows from \cref{asm:boundedness}.
    By taking $m\coloneqq\min\{k,\ceil{D\sqrt{L/(2n\delta)}}\}$, we have
    \begin{align}
        \Delta_{F_k}(x_k,u)
        &\leq \frac{LD^2}{2n}\max\set*{\frac1k,\frac1D\sqrt{\frac{2n\delta}L}}+\delta\min\set*{k+1,D\sqrt{\frac L{2n\delta}}+2}\\
        &\leq \frac{LD^2}{2n}\paren*{\frac1k+\frac1D\sqrt{\frac{2n\delta}L}}+\delta\paren*{D\sqrt{\frac L{2n\delta}}+2}\\
        &\leq \frac{LD^2}{2nk}+D\sqrt{\frac{2L\delta}n}+2\delta.
    \end{align}
    Taking $u\in\mathcal X_{F_k}^*$ completes the proof.
\end{proof}

\subsection{Proof of \cref{thm:tvpgd_lower_bound}}\label{sec:tvpgd_lower_bound_proof}
\begin{proof}
    Consider a one-dimensional constrained time-varying problem.
    Without loss of generality, let $x_0=D/2$.
    Let $\mathcal X\coloneqq[-D/2,D/2]$, $m\coloneqq\floor{(D/2)\sqrt{L/(n\delta)}}$, $K\coloneqq4m+\ceil*{LD^2/(4mn\delta)}$, 
    \begin{align}
        f_k(x)
        &\coloneqq\frac L2(\min\{x,0\})^2+a_kx,\quad
        a_k\coloneqq\frac{2\delta}D\cdot\begin{dcases*}
            k,&$0\leq k<m$,\\
            2m-k,&$m\leq k<3m$,\\
            -m,&$3m\leq k<K-m$,\\
            k-K,&$K-m\leq k<K$
        \end{dcases*}
    \end{align}
    for $0\leq k<K$, and periodically extend $f_k$ for $k\geq K$.
    By the definition of $f_k$, \cref{asm:function_class} holds.
    Since $|a_{k+1}-a_k|\leq 2\delta/D$, \cref{asm:t_lipschitz} also holds.
    It is sufficient to show that $x_K=D/2$ and $\Delta_m\geq \tfrac D8\sqrt{L\delta/n}$, because $f_k$ is periodic with period $K$.

    First, we show that $\Delta_m\geq \tfrac D8\sqrt{L\delta/n}$.
    For $0\leq k<m$, we have $L^{-1}\nabla f_k(x)=2\delta k/LD$ for all $x\geq0$.
    Thus, we have
    \begin{align}
        \sum_{k=0}^{m-1}\sum_{i=0}^{n-1}\frac{2\delta k}{LD}
        &=\frac{\delta nm(m-1)}{LD}
        \leq\frac{\delta n}{LD}\frac{LD^2}{4n\delta}
        =\frac D4
        <\frac D2,
    \end{align}
    where the first inequality follows from $m\leq \tfrac D2\sqrt{L/(n\delta)}$.
    This implies that 
    \begin{align}
        x_m
        &=\frac D2-\frac{\delta nm(m-1)}{LD}
        >\frac D4.
    \end{align}
    Hence, we have
    \begin{align}
        f_m(x_m)
        &=\frac{2m\delta}{D}x_m
        >\frac{m\delta}{2}.
    \end{align}
    From $\delta<LD^2/(16n)$, we have $\tfrac D2\sqrt{L/(n\delta)}\geq2$, which implies $m\geq \tfrac D4\sqrt{L/(n\delta)}$.
    Thus, we have
    \begin{align}
        f_m(x_m)
        >\frac D8\sqrt{\frac{L\delta}n}.
    \end{align}
    Since $\min_{x\in\mathcal X}f_m(x)\leq f_m(0)=0$, we have $\Delta_m\geq \tfrac D8\sqrt{L\delta/n}$.

    Next, we show that $x_K=D/2$.
    Let $\proj_{\mathcal X}(x)\coloneqq\argmin_{y\in \mathcal X}\|y-x\|$. 
    For $3m\leq k<K-m$, we have 
    \begin{align}
        \proj_{\mathcal X}(x-L^{-1}\nabla f_k(x))
        &=\proj_{\mathcal X}\paren*{x-\min\{x,0\}+\frac{2\delta m}{LD}}\\
        &=\min\set*{\frac D2,\max\{0,x\}+\frac{2\delta m}{LD}}
        \geq0
    \end{align}
    for all $x\in\mathcal X$.
    Thus, we have 
    \begin{align}
        x_{K-m}
        &=\min\set*{\frac D2,\max\{0,x_{3m}\}+\frac{2\delta mn}{LD}(K-4m)}
        \geq\min\set*{\frac D2,\frac{2\delta mn}{LD}\frac{LD^2}{4mn\delta}}
        =\frac D2,
    \end{align}
    where we used $K-4m=\ceil*{LD^2/(4mn\delta)}\geq LD^2/(4mn\delta)$ in the inequality.
    Since $a_k\leq0$ for $K-m\leq k<K$, we obtain $x_K=D/2$.
\end{proof}

\subsection{Example of Divergence When the Domain Is Unbounded}\label{sec:necessity_example}

\begin{example}
    Let $n=1$ and $g=0$ for simplicity.
    Then, the update rule of TV-PG is reduced to $x_{k+1}=x_k-L^{-1}\nabla f_k(x_k)$.
    Consider the following example:
    \begin{align}
        \min_{x\in\R}f_k(x)\coloneqq\frac12\sqrt{\delta^2+(\sqrt{2L\delta}x-2\delta k)^2}.
    \end{align}
    A simple calculation yields
    \begin{align}
        \nabla f_k(x)
        &= \frac{\sqrt{2L\delta}(\sqrt{2L\delta}x-2\delta k)}{2\sqrt{\delta^2+(\sqrt{2L\delta}x-2\delta k)^2}}
    \end{align}
    and
    \begin{align}
        \nabla^2f_k(x)
        =\frac {L\delta^3}{(\delta^2+(\sqrt{2L\delta}x-2\delta k)^2)^{3/2}}
        \in (0,L],
    \end{align}
    which shows that \cref{asm:function_class} holds.
    By the mean value theorem, there exists $y\in(\sqrt{2L\delta}x-2\delta(k+1),\sqrt{2L\delta}x-2\delta k)$ such that
    \begin{align}
        |f_{k+1}(x)-f_k(x)|
        &=\abs{\frac y{2\sqrt{\delta^2+y^2}}\cdot 2\delta}
        \leq \delta,
    \end{align}
    which implies that \cref{asm:t_lipschitz} holds.
    Next, we bound the distance between consecutive iterates.
    \begin{align}
        |x_{k+1}-x_k|
        &=|L^{-1}\nabla f_k(x_k)| 
        \leq \frac{\sqrt{2L\delta}|\sqrt{2L\delta}x_k-2\delta k|}{2L\sqrt{\delta^2+(\sqrt{2L\delta}x_k-2\delta k)^2}} 
        \leq \sqrt{\frac{\delta}{2L}},
    \end{align}
    and thus $|x_k-x_0|\leq k\sqrt{\delta/(2L)}$.
    Let $x_0=0$ for simplicity.
    Then, we have
    \begin{align}
        \Delta_k
        &=\frac12\sqrt{\delta^2+(\sqrt{2L\delta}x_k-2\delta k)^2}-\frac\delta2\\
        &\geq\frac{\abs*{2\delta k-\sqrt{2L\delta}x_k}}2-\frac\delta2\\
        &\geq \delta k-\frac{\sqrt{2L\delta}}2\cdot\abs{x_k}-\frac\delta2\\
        &\geq \frac\delta2 k-\frac\delta2
        \to\infty\quad\text{as }k\to\infty,
    \end{align}
    where we used the triangle inequality in the second inequality and $|x_k|\leq k\sqrt{\delta/(2L)}$ in the third inequality.
\end{example}

Note that this phenomenon does not occur if we assume \cref{asm:problem} because \cref{thm:tvpgd_new_error} guarantees that $\Delta_k$ is bounded.

%% file: 92_appendix_apg.tex
\section{DETAILS FOR SECTION~\ref{sec:apg}}\label{sec:apg_proofs}
\subsection{Equivalence of Update Rules for APG${}_1$}\label{sec:fista_d_update_equivalence_proof}
We formally state the equivalence between the update rules \cref{eq:apg1_equivalence_gamma,eq:apg1_equivalence_y,eq:apg1_equivalence_x} and \cref{alg:fista} as follows.

\begin{proposition}\label{prop:apg1_equivalence}
    Let $g\in\mathcal G$ satisfy \cref{asm:boundedness} and $x, z\in\mathcal X$ and $x^-\in\R^d$ satisfy $z=\gamma_{-} x-(\gamma_{-}-1)x^{-}$, $\gamma_-\geq0$, $\gamma=\phi(\gamma_-)$, and $f\in\mathcal F_L$.
    Then, the updates \cref{eq:apg1_equivalence_y}, \cref{eq:apg1_equivalence_x}, and $z^+=\gamma x^+-(\gamma-1)x$ are equivalent to
    \begin{align}
        y&=(1-\gamma^{-1})x+\gamma^{-1}z,\label{eq:apg1_y}\\
        z^{+}&=\prox_{\gamma L^{-1}g}(z-\tfrac\gamma L\nabla f(y)),\label{eq:apg1_z}\\
        x^{+}&=(1-\gamma^{-1})x+\gamma^{-1}z^{+}.\label{eq:apg1_x}
    \end{align}
\end{proposition}

\begin{proof}
    First, we prove that \cref{eq:apg1_y,eq:apg1_z,eq:apg1_x} imply \cref{eq:apg1_equivalence_y}, \cref{eq:apg1_equivalence_x} and $z^+=\gamma x^+-(\gamma-1)x$.
    By substituting $z=\gamma_{-} x-(\gamma_{-}-1)x^{-}$ into \cref{eq:apg1_y}, we have \cref{eq:apg1_equivalence_y} as follows:
    \begin{align}
        y
        &=(1-\gamma^{-1})x+\gamma^{-1}(\gamma_{-}x-(\gamma_{-}-1)x^{-})\\
        &=x+\frac{\gamma_{-}-1}\gamma(x-x^{-}).
    \end{align}
    The identity \cref{eq:apg1_x} gives $z^+=\gamma x^+-(\gamma-1)x$.
    Finally, \cref{eq:apg1_equivalence_x} can be derived as follows:
    \begin{align}
        x^{+}
        &=(1-\gamma^{-1})x+\gamma^{-1}\prox_{\gamma L^{-1}g}\paren*{z-\frac\gamma L\nabla f(y)}\\
        &=(1-\gamma^{-1})x+\gamma^{-1}\prox_{\gamma L^{-1}g}\paren*{\gamma y-(\gamma-1)x-\frac\gamma L\nabla f(y)}\\
        &=(1-\gamma^{-1})x+\gamma^{-1}\argmin_{w\in\R^d}\set*{g(w)+\frac L{2\gamma}\norm{w-\gamma y+(\gamma-1)x+\frac\gamma L\nabla f(y)}^2}\\
        &=\argmin_{w'\in\R^d}\set*{\gamma^{-1}g(\gamma w'-(\gamma-1)x)+\frac L2\norm{w'-y+L^{-1}\nabla f(y)}^2}\\
        &=T_{f,\tilde g_{x,\gamma}}(y),
    \end{align}
    where the first equality follows from \cref{eq:apg1_z,eq:apg1_x} and the second equality follows from \cref{eq:apg1_y}.
    We also use the change of variables $w'=(1-\gamma^{-1})x+\gamma^{-1}w$.

    We next prove the converse, namely that \cref{eq:apg1_equivalence_y,eq:apg1_equivalence_x} and $z^+=\gamma x^+-(\gamma-1)x$ imply \cref{eq:apg1_y,eq:apg1_z,eq:apg1_x}.
    From \cref{eq:apg1_equivalence_y} and $z=\gamma_{-} x-(\gamma_{-}-1)x^{-}$, we have 
    \begin{align}
        y
        &=x+\frac{\gamma_{-}-1}\gamma(x-x^{-})\\
        &=(1-\gamma^{-1})x+\gamma^{-1}z,
    \end{align}
    which gives \cref{eq:apg1_y}.
    The identity $z^+=\gamma x^+-(\gamma-1)x$ gives \cref{eq:apg1_x}.
    It therefore remains to show \cref{eq:apg1_z}.
    Repeating the preceding calculation gives
    \begin{align}
        (1-\gamma^{-1})x+\gamma^{-1}\prox_{\gamma L^{-1}g}\paren*{z-\frac\gamma L\nabla f(y)}
        &=T_{f,\tilde g_{x,\gamma}}(y)
        =x^+
        =(1-\gamma^{-1})x+\gamma^{-1}z^{+},
    \end{align}
    where the second equality follows from \cref{eq:apg1_equivalence_x} and the third equality follows from \cref{eq:apg1_x}.
    This implies that \cref{eq:apg1_z} holds.
\end{proof}

\subsection{Proof of \cref{prop:apg2_equivalence}}\label{sec:apg2_equivalence_proof}
\begin{proof}
    First, we prove that \cref{eq:apg2_beta,eq:apg2_y,eq:apg2_z,eq:apg2_x} imply \cref{eq:apg2_equivalence_beta,eq:apg2_equivalence_y,eq:apg2_equivalence_x} and $z^+=2x^+-x$.
    The identity \cref{eq:apg2_x} gives $z^+=2x^+-x$.
    Since $x^-=2x-z$, we have $x^--x=x-z$.
    Combining this with \cref{eq:apg2_beta,eq:apg2_y} gives \cref{eq:apg2_equivalence_beta,eq:apg2_equivalence_y}.
    To derive the remaining update for $x^+$, we write
    \begin{align}
        x^+
        =\frac{x}2+\frac{z^+}2
        &=\frac {x}2+\frac12\prox_{2L^{-1}g}\paren*{2y-x-\frac2L\nabla f(y)}\\
        &=\frac {x}2+\frac12\argmin_{z\in\R^d}\set*{g(z)+\frac L4\norm{z-2y+x+\frac2L\nabla f(y)}^2}\\
        &=\argmin_{z'\in\R^d}\set*{\frac12g(2z'-x)+\frac L2\norm{z'-y+L^{-1}\nabla f(y)}^2}\\
        &=T_{f,\tilde g_{x,2}}(y),
    \end{align}
    where the first and second equalities use \cref{eq:apg2_x} and \cref{eq:apg2_z}, respectively.
    We also use the change of variables $z'=(z+x)/2$.

    We next prove the converse, namely that \cref{eq:apg2_equivalence_beta,eq:apg2_equivalence_y,eq:apg2_equivalence_x} and $z^+=2x^+-x$ imply \cref{eq:apg2_beta,eq:apg2_y,eq:apg2_z,eq:apg2_x}.
    The identity $x^-=2x-z$ gives $x^--x=x-z$.
    Combining this with \cref{eq:apg2_equivalence_beta,eq:apg2_equivalence_y} gives \cref{eq:apg2_beta,eq:apg2_y}.
    The identity $z^+=2x^+-x$ gives \cref{eq:apg2_x}.
    It therefore remains to show \cref{eq:apg2_z}.
    Repeating the preceding calculation gives
    \begin{align}
        \frac {x}2+\frac12\prox_{2L^{-1}g}\paren*{2y-x-\frac2L\nabla f(y)}
        &=T_{f,\tilde g_{x,2}}(y)
        =x^+
        =\frac{x}2+\frac{z^+}2,
    \end{align}
    where the second and third equalities follow from \cref{eq:apg2_equivalence_x} and $z^+=2x^+-x$, respectively.
    This implies that \cref{eq:apg2_z} holds.
\end{proof}

\subsection{Proof of \cref{prop:fista_new_potential_decrease}}\label{sec:fista_new_potential_decrease_proof}
Before proving \cref{prop:fista_new_potential_decrease}, we introduce lemmas.
The following lemma is an analogue of \cref{prop:pgd_property} for the $\apg_2$ operator.

\begin{lemma}\label{lem:fista_property}
    Let $g\in\mathcal G$ be a function satisfying \cref{asm:boundedness}, $f\in\mathcal F_L$, $x,z,u\in\mathcal X$, and $\alpha\in[0,\tfrac12]$, and let $(x^+,z^+)\coloneqq S_{f}(x,z)$.
    Then, we have 
    \begin{align}
        &\Delta_F(x^+,u)+\frac L2\norm{\alpha(x^+-u)+(1-\alpha)(z^+-x^+)}^2\\
        &\leq(1-\alpha)\Delta_F(x,u)+\frac L2\norm{\alpha(x-u)+\beta(z-x)}^2,
    \end{align}
    where $\beta\coloneqq(1+\tfrac2D\norm{z-x})^{-1}$.
\end{lemma}

\begin{proof}
    From \cref{prop:apg2_equivalence}, we have $x^+=T_{f,\tilde g_{x,2}}(y)$, where $y\coloneqq(1-\beta)x+\beta z$.
    Thus, we have
    \begin{align}
        x^+
        &=\argmin_{z'\in\R^d}\set*{\tilde g_{x,2}(z')+\frac L2\norm{z'-y+L^{-1}\nabla f(y)}^2}\\
        &=\argmin_{z'\in\R^d}\set*{\tilde g_{x,2}(z')+f(y)+\inner{\nabla f(y),z'-y}+\frac L2\norm{z'-y}^2}
        =\argmin_{z'\in\R^d}\tilde F_{x,y}(z'),
    \end{align}
    where $\tilde F_{x,y}(z')\coloneqq \tilde g_{x,2}(z')+f(y)+\inner{\nabla f(y),z'-y}+\tfrac L2\norm{z'-y}^2$.
    Since $\tilde F_{x,y}$ is $L$-strongly convex, we have
    \begin{align}
        \tilde F_{x,y}(x^+)+\frac L2\norm{x^+-(1-\alpha)x-\alpha u}^2
        &\leq\tilde F_{x,y}((1-\alpha)x+\alpha u).\label{eq:fista_property_1}
    \end{align}
    
    Next, we give a lower bound of the left-hand side of \cref{eq:fista_property_1}.
    Since $f$ is $L$-smooth and $g(z')\leq\tilde g_{x,2}(z')$ by the convexity of $g$, we have $F(x^+)\leq\tilde F_{x,y}(x^+)$.
    By the definition of $S_f$, we have $z^+=2x^+-x$, and thus we have
    \begin{align}
        x^+-(1-\alpha)x-\alpha u
        &=\alpha(x^+-u)+(1-\alpha)(z^+-x^+).
    \end{align}
    Combining these inequalities gives
    \begin{align}
        F(x^+)+\frac L2\norm{\alpha(x^+-u)+(1-\alpha)(z^+-x^+)}^2
        &\leq\tilde F_{x,y}(x^+)+\frac L2\norm{x^+-(1-\alpha)x-\alpha u}^2.\label{eq:fista_property_2}
    \end{align}

    Finally, we give an upper bound of the right-hand side of \cref{eq:fista_property_1}.
    By the definition of $\tilde F_{x,y}$, we have
    \begin{align}
        &\tilde F_{x,y}((1-\alpha)x+\alpha u)\\
        &=\tilde g_{x,2}((1-\alpha)x+\alpha u)+\Big(f(y)+\inner{\nabla f(y),(1-\alpha)x+\alpha u-y}\Big)+\frac L2\norm{(1-\alpha)x+\alpha u-y}^2.\label{eq:fista_property_3}
    \end{align}
    For the first term in \cref{eq:fista_property_3}, the convexity of $g$ gives
    \begin{align}
        \tilde g_{x,2}((1-\alpha)x+\alpha u)
        &=\frac12g(x)+\frac12g((1-2\alpha)x+2\alpha u)\\
        &\leq\frac12g(x)+\frac{1-2\alpha}2g(x)+\alpha g(u)\\
        &=(1-\alpha)g(x)+\alpha g(u).
    \end{align}
    For the second term in \cref{eq:fista_property_3}, the convexity of $f$ gives
    \begin{align}
        f(y)+\inner{\nabla f(y),(1-\alpha)x+\alpha u-y}
        \leq f((1-\alpha)x+\alpha u)
        &\leq(1-\alpha)f(x)+\alpha f(u).
    \end{align}
    For the third term in \cref{eq:fista_property_3}, the definition of $y$ gives
    \begin{align}
        (1-\alpha)x+\alpha u-y
        &=-\alpha(x-u)-\beta(z-x).
    \end{align}
    Combining these relations gives
    \begin{align}
        \tilde F_{x,y}((1-\alpha)x+\alpha u)
        &\leq(1-\alpha)F(x)+\alpha F(u)+\frac L2\norm{\alpha(x-u)+\beta(z-x)}^2.\label{eq:fista_property_4}
    \end{align}

    Combining \cref{eq:fista_property_1,eq:fista_property_2,eq:fista_property_4} gives the desired inequality.
\end{proof}

\begin{lemma}\label{lem:geometric_inequality}
    Let $v,w\in\R^d$, $p\geq5$, and $\beta\coloneqq(1+\tfrac2D\norm{w})^{-1}$.
    Suppose $\norm{v},\norm{w}\leq D$.
    Then, we have
    \begin{align}
        \norm{3v+(p+1)\beta w}^2
        &\leq\theta\norm{3v+(p-3)w}^2+60(1-\theta)D^2,
    \end{align}
    where $\theta\coloneqq(p+1)(p-2)p^{-2}$.
\end{lemma}

\begin{proof}
    Suppose $D=1$ without loss of generality.
    Let 
    \begin{align}
        E\coloneqq\norm{3v+(p+1)\beta w}^2-\theta\norm{3v+(p-3)w}^2.
    \end{align}
    It is sufficient to show that $E\leq60(1-\theta)$.
    First, expanding the norms gives
    \begin{align}
        E
        &=9(1-\theta)\norm{v}^2+6A\inner{v,w}+B\norm{w}^2,
    \end{align}
    where $A\coloneqq(p+1)\beta-(p-3)\theta$ and $B\coloneqq(p+1)^2\beta^2-(p-3)^2\theta$.
    For the second term, the Cauchy--Schwarz inequality and the AM-GM inequality give
    \begin{align}
        6A\inner{v,w}
        &\leq6|A|\norm{v}\norm{w}
        \leq\frac{9\norm{v}^2\norm{w}}{1+\norm{w}}+A^2\norm{w}(1+\norm{w}).
    \end{align}
    Thus, we have
    \begin{align}
        E
        &\leq9(1-\theta)+\frac{9\norm{w}}{1+\norm{w}}+\norm{w}\Big(A^2(1+\norm{w})+B\norm{w}\Big),\label{eq:geometric_inequality_1}
    \end{align}
    where we used $\norm{v}\leq 1$.

    Next, we give an upper bound of the last term in \cref{eq:geometric_inequality_1}.
    Let $C\coloneqq p+1-\theta(p-3)$.
    By definition, we have
    \begin{align}
        C
        &=p+1-(p+1)(p-2)(p-3)p^{-2}
        =(5p^2-p-6)p^{-2}
        \in[4,5],
    \end{align}
    where we used $p\geq5$ for the lower bound.
    By using $C$ and the relation $1=\beta(1+2\norm{w})$, the coefficients $A$ and $B$ can be rewritten as follows:
    \begin{align}
        A
        &=\beta\Big(p+1-\theta(p-3)(1+2\norm{w})\Big)
        =\beta\Big(C-2\theta(p-3)\norm{w}\Big)
    \end{align}
    and
    \begin{align}
        B
        &=(C+\theta(p-3))^2\beta^2-(p-3)^2\theta.
    \end{align}
    Thus, by an algebraic manipulation, we have
    \begin{align}
        &A^2(1+\norm{w})+B\norm{w}\\
        &=\beta^2\Big(C-2\theta(p-3)\norm{w}\Big)^2(1+\norm{w})+\Big((C+\theta(p-3))^2\beta^2-(p-3)^2\theta\Big)\norm{w}\\
        &=\beta^2\Big(C^2-4\theta(p-3)\norm{w}C+4\theta^2(p-3)^2\norm{w}^2\Big)(1+\norm{w})\\&\quad+\beta^2\Big(C^2+2\theta(p-3)C+\theta^2(p-3)^2\Big)\norm{w}-(p-3)^2\theta\norm{w}\\
        &=\beta^2\Big((1+2\norm{w})C^2-2\theta(p-3)(1+2\norm{w})\norm{w}C+(1+2\norm{w})^2\theta^2(p-3)^2\norm{w}\Big)\\&\quad-(p-3)^2\theta\norm{w}.
    \end{align}
    Using $\beta(1+2\norm{w})=1$ again simplifies this expression as follows:
    \begin{align}
        A^2(1+\norm{w})+B\norm{w}
        &=\beta C^2-2\beta\theta(p-3)\norm{w}C-\theta(1-\theta)(p-3)^2\norm{w}.
    \end{align}
    Combining this identity with \cref{eq:geometric_inequality_1} gives
    \begin{align}
        \frac E{1-\theta}
        &\leq9+\frac{\norm{w}}{1-\theta}\paren*{\frac{9}{1+\norm{w}}+\beta C^2}-\frac{2\beta\theta(p-3)\norm{w}^2C}{1-\theta}-\theta(p-3)^2\norm{w}^2\\
        &\leq9+\frac{\norm{w}}{1-\theta}\paren*{\frac{9}{1+\norm{w}}+25\beta}-\frac{8\beta\theta(p-3)\norm{w}^2}{1-\theta},\label{eq:geometric_inequality_2}
    \end{align}
    where $C\in[4,5]$ is used for the last inequality.

    Finally, we show the right-hand side of \cref{eq:geometric_inequality_2} is bounded by $60$.
    For the second term, we have
    \begin{align}
        \frac{9}{1+\norm{w}}
        &=\frac{9}{1+\norm{w}}-\frac{15}{1+2\norm{w}}+15\beta
        =\frac{3\norm{w}-6}{(1+\norm{w})(1+2\norm{w})}+15\beta
        \leq15\beta,
    \end{align}
    where the first equality follows from $\beta(1+2\norm{w})=1$ and the inequality follows from $\norm{w}\leq1$.
    In addition, we have
    \begin{align}
        \frac1{1-\theta}
        &=\frac{p^2}{p+2}
        =p-2+\frac{4}{p+2}
        \leq p-1,
    \end{align}
    where the inequality follows from $p\geq5$.
    For the third term, we have
    \begin{align}
        \frac{\theta}{1-\theta}
        &=\frac{(p+1)(p-2)}{p+2}
        =p-3+\frac{4}{p+2}
        \geq p-3.
    \end{align}
    Combining these inequalities with \cref{eq:geometric_inequality_2} gives
    \begin{align}
        \frac E{1-\theta}
        &\leq9+40\beta(p-1)\norm{w}-8\beta(p-3)^2\norm{w}^2\\
        &=9+80\beta\norm{w}+40\beta(p-3)\norm{w}-8\beta(p-3)^2\norm{w}^2\\
        &=9+80\beta\norm{w}+50\beta-8\beta\paren*{(p-3)\norm{w}-\frac52}^2\\
        &\leq9+\beta\paren{50+80\norm{w}}\\
        &\leq60,
    \end{align}
    where $\beta(1+2\norm{w})=1$ is used for the last inequality.
\end{proof}

Using these lemmas, we can prove \cref{prop:fista_new_potential_decrease} as follows.

\begin{proof}[Proof of \cref{prop:fista_new_potential_decrease}]
    Let $(x^+,z^+)\coloneqq S_f(x,z)$ and $\beta\coloneqq(1+\tfrac2D\norm{z-x})^{-1}$.
    \cref{lem:fista_property} with $\alpha=3/(p+1)\in[0,\tfrac12]$ gives
    \begin{align}
        &\Delta_F(x^+,u)+\frac L2\norm{\frac3{p+1}(x^+-u)+\frac{p-2}{p+1}(z^+-x^+)}^2\\
        &\leq\frac{p-2}{p+1}\Delta_F(x,u)+\frac L2\norm{\frac3{p+1}(x-u)+\beta(z-x)}^2.
    \end{align}
    Multiplying both sides by $(p+1)^2$ gives
    \begin{align}
        \Phi_{F,p+1}(x^+,z^+,u)
        &=(p+1)^2\Delta_F(x^+,u)+\frac L2\norm{3(x^+-u)+(p-2)(z^+-x^+)}^2\\
        &\leq\theta p^2\Delta_F(x,u)+\frac L2\norm{3(x-u)+(p+1)\beta(z-x)}^2\\
        &\leq\theta p^2\Delta_F(x,u)+\frac {\theta L}2\norm{3(x-u)+(p-3)(z-x)}^2+30(1-\theta)LD^2\\
        &=\theta\Phi_{F,p}(x,z,u)+30(1-\theta)LD^2,
    \end{align}
    where the second inequality follows from \cref{lem:geometric_inequality} with $v=x-u$ and $w=z-x$, which completes the proof.
\end{proof}

%% file: 92_appendix_tvapg.tex
\section{DETAILS FOR SECTION~\ref{sec:tvapg}}\label{sec:tvapg_proofs}
\subsection{Uniform Bound on the Potential Function}\label{sec:uniform_potential_bound}
This section provides a uniform bound on the potential function.

First, we show a recursive inequality for the function value gap.

\begin{lemma}\label{lem:fista_loose_recursion}
    Let $g\in\mathcal G$ be a function satisfying \cref{asm:boundedness}, $f\in\mathcal F_L$, $x,z,u\in\mathcal X$, and $(x^+,z^+)\coloneqq S_{f}(x,z)$.
    Then, we have 
    \begin{align}
        \Delta_F(x^+,u)
        &\leq\frac12\Delta_F(x,u)+\frac{LD^2}2.
    \end{align}
\end{lemma}

\begin{proof}
    Let $\beta\coloneqq(1+\tfrac2D\norm{z-x})^{-1}$.
    From \cref{lem:fista_property} with $\alpha=\tfrac12$, we have
    \begin{align}
        \Delta_F(x^+,u)
        &\leq\frac12\Delta_F(x,u)+\frac L2\norm{\frac12(x-u)+\beta(z-x)}^2\\
        &=\frac12\Delta_F(x,u)+\frac L2\norm{((1-\beta)x+\beta z)-\frac12(x+u)}^2\\
        &\leq\frac12\Delta_F(x,u)+\frac {LD^2}2,
    \end{align}
    where the last inequality follows from \cref{asm:boundedness}.
\end{proof}

We next establish a uniform bound on the function value gap after the first iteration of each round of \cref{alg:tvfista_new}.

\begin{lemma}\label{lem:uniform_bound}
    Suppose \cref{asm:problem} holds.
    Let $x_0,u\in\mathcal X$, $z_0\coloneqq x_0$, and $n\geq1$, and let $(x_k,z_k)_{k\geq1}\subset \mathcal X^2$ be a sequence generated by \cref{alg:tvfista_new}.
    Let
    \begin{align}
        x_0^+
        \coloneqq z_0^+
        &\coloneqq T_{f_0,g}(x_0)\quad\text{and}\quad
        (x_{k}^+,z_{k}^+)
        \coloneqq S_{f_k}(x_k,z_k)\quad\text{for }k\geq1.
    \end{align}
    Then, we have 
    \begin{align}
        \Delta_{F_k}(x_{k}^+,u)
        &\leq LD^2+2\delta
    \end{align}
    for all $k\geq0$.
\end{lemma}

\begin{proof}
    We prove the lemma by induction.
    For $k=0$, we have
    \begin{align}
        \Delta_{F_0}(x_0^+,u)
        &=\Delta_{F_0}(T_{f_0,g}(x_0),u)
        \leq\frac L2\norm{x_0-u}^2
        \leq LD^2,
    \end{align}
    where the first inequality follows from \cref{prop:pgd_property} and the second inequality follows from \cref{asm:boundedness}.

    Assume that the bound holds for some $k\geq0$.
    Since $(x_{k+1},z_{k+1})=S_{f_k}^{n-1}(x_k^+,z_k^+)$, repeated application of \cref{lem:fista_loose_recursion} over the remaining $n-1$ iterations gives
    \begin{align}
        \Delta_{F_k}(x_{k+1},u)
        &\leq2^{-n+1}\Delta_{F_k}(x_k^+,u)+\paren*{1-2^{-n+1}}LD^2\\
        &\leq2^{-n+1}(LD^2+2\delta)+\paren*{1-2^{-n+1}}LD^2\\
        &\leq LD^2+2\delta,
    \end{align}
    where the second inequality follows from the induction hypothesis.
    Combining this with \cref{asm:t_lipschitz}, we have
    \begin{align}
        \Delta_{F_{k+1}}(x_{k+1},u)
        &\leq\Delta_{F_k}(x_{k+1},u)+2\delta
        \leq LD^2+4\delta.\label{eq:uniform_bound_1}
    \end{align}
    Since $(x_{k+1}^+,z_{k+1}^+)=S_{f_{k+1}}(x_{k+1},z_{k+1})$, \cref{lem:fista_loose_recursion} implies
    \begin{align}
        \Delta_{F_{k+1}}(x_{k+1}^+,u)
        &\leq\frac12\Delta_{F_{k+1}}(x_{k+1},u)+\frac {LD^2}2
        \leq LD^2+2\delta,
    \end{align}
    where the last inequality follows from \cref{eq:uniform_bound_1}, which completes the induction.
\end{proof}

Next, we derive a recursive inequality for the potential function $\Phi$ using \cref{prop:fista_new_potential_decrease}.

\begin{lemma}\label{lem:fista_new_potential_decrease_iterative}
    Let $g\in\mathcal G$ be a function satisfying \cref{asm:boundedness}, $f\in\mathcal F_L$, $x,z,u\in\mathcal X$, $n\geq0$, and $p\geq5$.
    Then, we have 
    \begin{align}
        \rho_{p+n}\Phi^{(n)}
        &\leq\rho_p\Phi^{(0)}+30LD^2\paren*{n-\frac2p+\frac2{p+n}},
    \end{align}
    where $\Phi^{(i)}\coloneqq\Phi_{F,p+i}(S_{f}^i(x,z),u)$ and $\rho_i\coloneqq(i-1)(i-2)/i$.
\end{lemma}

\begin{proof}
    Since the case $n=0$ is trivial, we assume $n\geq1$.
    \cref{prop:fista_new_potential_decrease} implies
    \begin{align}
        \Phi^{(i+1)}
        &\leq\frac{(p+i+1)(p+i-2)}{(p+i)^2}\Phi^{(i)}+30LD^2\frac{p+i+2}{(p+i)^2}
    \end{align}
    for all $i\geq0$.
    Multiplying both sides by $(p+i)(p+i-1)/(p+i+1)$ gives
    \begin{align}
        \frac{(p+i)(p+i-1)}{p+i+1}\Phi^{(i+1)}
        &\leq\frac{(p+i-1)(p+i-2)}{p+i}\Phi^{(i)}+30LD^2\frac{(p+i+2)(p+i-1)}{(p+i)(p+i+1)}\\
        &=\frac{(p+i-1)(p+i-2)}{p+i}\Phi^{(i)}+30LD^2\paren*{1-\frac2{p+i}+\frac2{p+i+1}}.
    \end{align}
    Summing both sides from $i=0$ to $i=n-1$ gives
    \begin{align}
        \frac{(p+n-1)(p+n-2)}{p+n}\Phi^{(n)}
        &\leq\frac{(p-1)(p-2)}{p}\Phi^{(0)}+30LD^2\paren*{n-\frac2p+\frac2{p+n}},
    \end{align}
    which completes the proof.
\end{proof}

Finally, we show a uniform bound on the potential $\Phi_{F_k,4+n}(x_{k},z_{k},u)$ for all $k\geq1$.

\begin{lemma}\label{lem:uniform_potential_bound}
    Suppose \cref{asm:problem} holds.
    Let $x_0,u\in\mathcal X$, $z_0\coloneqq x_0$, and $n\geq1$, and let $(x_k,z_k)_{k\geq1}\subset \mathcal X^2$ be a sequence generated by \cref{alg:tvfista_new}.
    Then, we have 
    \begin{align}
        \rho_{4+n}\Phi_{F_k,4+n}(x_{k},z_{k},u)
        &\leq30LD^2\paren*{n+1+\frac2{4+n}}+120\delta+2\delta \rho_{4+n}(4+n)^2
    \end{align}
    for all $k\geq1$, where $\rho_i\coloneqq(i-1)(i-2)/i$.
\end{lemma}

\begin{proof}
    Let $x_0^+\coloneqq z_0^+\coloneqq T_{f_0,g}(x_0)$ and $(x_{k}^+,z_{k}^+)\coloneqq S_{f_k}(x_k,z_k)$ for all $k\geq1$.
    Since $(x_{k+1},z_{k+1})=S_{f_k}^{n-1}(x_{k}^+,z_{k}^+)$, \cref{lem:fista_new_potential_decrease_iterative} with $p=5$ implies
    \begin{align}
        \rho_{4+n}\Phi_{F_k,4+n}(x_{k+1},z_{k+1},u)
        &\leq\rho_5\Phi_{F_k,5}(x_k^+,z_k^+,u)+30LD^2\paren*{n-1-\frac25+\frac2{4+n}}\\
        &=\frac{12}5\Phi_{F_k,5}(x_k^+,z_k^+,u)+30LD^2\paren*{n-\frac75+\frac2{4+n}}
    \end{align}
    for all $k\geq0$.
    From \cref{lem:uniform_bound}, we have
    \begin{align}
        \Phi_{F_k,5}(x_k^+,z_k^+,u)
        &\leq25(LD^2+2\delta)+\frac L2\norm{3(x_k^+-u)+2(z_k^+-x_k^+)}^2\\
        &=25LD^2+50\delta+\frac L2\norm{(x_k^+-u)+2(z_k^+-u)}^2\\
        &\leq30LD^2+50\delta
    \end{align}
    for all $k\geq0$, where the second inequality follows from \cref{asm:boundedness}.
    From \cref{asm:t_lipschitz}, we have
    \begin{align}
        \Phi_{F_{k+1},4+n}(x_{k+1},z_{k+1},u)
        &\leq\Phi_{F_k,4+n}(x_{k+1},z_{k+1},u)+2\delta (4+n)^2
    \end{align}
    for all $k\geq0$.
    Combining the above three inequalities gives 
    \begin{align}
        \rho_{4+n}\Phi_{F_{k+1},4+n}(x_{k+1},z_{k+1},u)
        &\leq\frac{12}5\paren*{30LD^2+50\delta}+30LD^2\paren*{n-\frac75+\frac2{4+n}}+2\delta\rho_{4+n} (4+n)^2\\
        &=30LD^2\paren*{n+1+\frac2{4+n}}+120\delta+2\delta \rho_{4+n}(4+n)^2
    \end{align}
    for all $k\geq0$, which completes the proof.
\end{proof}

\subsection{Proof of \cref{thm:tvfista_new_error}}\label{sec:tvfista_new_error_proof}
\begin{proof}
    Fix $u\in\mathcal X$ arbitrarily.
    \cref{lem:fista_new_potential_decrease_iterative} implies
    \begin{align}
        \rho_{p+n}\Phi_{F_k,p+n}(x_{k+1},z_{k+1},u)
        &\leq\rho_p\Phi_{F_k,p}(x_k,z_k,u)+30LD^2\paren*{n-\frac2p+\frac2{p+n}}
    \end{align}
    for $k\geq1$, where $\rho_i\coloneqq(i-1)(i-2)/i$.
    From \cref{asm:t_lipschitz}, we have
    \begin{align}
        \Phi_{F_{k+1},p+n}(x_{k+1},z_{k+1},u)
        &\leq\Phi_{F_k,p+n}(x_{k+1},z_{k+1},u)+2\delta (p+n)^2.
    \end{align}
    Combining the above two inequalities gives
    \begin{align}
        \rho_{p+n}\Phi_{F_{k+1},p+n}(x_{k+1},z_{k+1},u)
        &\leq\rho_p\Phi_{F_k,p}(x_k,z_k,u)+30LD^2\paren*{n-\frac2p+\frac2{p+n}}+2\delta\rho_{p+n}(p+n)^2.
    \end{align}
    In particular, by taking $k\gets k-m+i$ and $p=4+ni$, we have
    \begin{align}
        &\rho_{4+n(i+1)}\Phi_{k-m+i+1,i+1}\\
        &\leq\rho_{4+ni}\Phi_{k-m+i,i}+30LD^2\paren*{n-\frac2{4+ni}+\frac2{4+n(i+1)}}+2\delta\rho_{4+n(i+1)}(4+n(i+1))^2
    \end{align}
    for all $1\leq i< m\leq k$, where $\Phi_{k,i}\coloneqq\Phi_{F_k,4+ni}(x_k,z_k,u)$.
    Summing both sides from $i=1$ to $i=m-1$ gives
    \begin{align}
        \rho_{4+nm}\Phi_{k,m}
        &\leq\rho_{4+n}\Phi_{k-m+1,1}+30LD^2\paren*{n(m-1)-\frac2{4+n}+\frac2{4+nm}}+2\delta\sum_{i=2}^{m}\rho_{4+ni}(4+ni)^2
    \end{align}
    for all $2\leq m\leq k$.
    The same inequality holds trivially for $m=1$.
    To bound the potential term on the right-hand side, we use the identity $\Phi_{k-m+1,1}=\Phi_{F_{k-m+1},4+n}(x_{k-m+1},z_{k-m+1},u)$.
    Applying \cref{lem:uniform_potential_bound} with $k\gets k-m+1\geq1$ then gives
    \begin{align}
        \rho_{4+n}\Phi_{k-m+1,1}
        &\leq30LD^2\paren*{n+1+\frac2{4+n}}+120\delta+2\delta \rho_{4+n}(4+n)^2
    \end{align}
    for all $1\leq m\leq k$. 
    Combining the above two inequalities gives
    \begin{align}
        \rho_{4+nm}\Phi_{k,m}
        &\leq30LD^2\paren*{nm+1+\frac2{4+nm}}+2\delta\sum_{i=1}^{m}\rho_{4+ni}(4+ni)^2+120\delta\\
        &\leq30(nm+2)LD^2+2\delta\sum_{i=1}^{m}\rho_{4+ni}(4+ni)^2+120\delta
    \end{align}
    for all $1\leq m\leq k$.
    Let $(i)_j\coloneqq i(i+1)\cdots(i+j-1)$ for $i,j\geq1$.
    Since $\Phi_{k,m}\geq(4+nm)^2\Delta_{F_k}(x_k,u)$ and $\rho_{4+nm}(4+nm)^2=(nm+2)_3$, we obtain
    \begin{align}
        \Delta_{F_k}(x_k,u)
        &\leq\frac1{(nm+2)_3}\paren*{30(nm+2)LD^2+2\delta\sum_{i=1}^{m}(ni+2)_3+120\delta}\\
        &\leq\frac{30LD^2}{n^2m^2}+2\delta\sum_{i=1}^{m}\frac{(ni+2)_3}{(nm+2)_3}+\frac{120\delta}{(nm+2)_3}\label{eq:tvfista_new_error_1}
    \end{align}
    for all $1\leq m\leq k$. 

    Since
    \begin{align}
        \frac{ni+\ell}{nm+\ell}
        &\leq\frac{i+\ell}{m+\ell}
    \end{align}
    for all $n\geq1$, $1\leq i\leq m$, and $\ell\geq0$, we have $(ni+2)_3/(nm+2)_3\leq(i+2)_3/(m+2)_3$.
    Thus, we have
    \begin{align}
        \sum_{i=1}^{m}\frac{(ni+2)_3}{(nm+2)_3}
        &\leq\sum_{i=1}^{m}\frac{(i+2)_3}{(m+2)_3}
        =\frac{(m+2)_4-120}{4(m+2)_3}
        =\frac{m+5}4-\frac{30}{(m+2)_3},
    \end{align}
    and hence
    \begin{align}
        2\delta\sum_{i=1}^{m}\frac{(ni+2)_3}{(nm+2)_3}+\frac{120\delta}{(nm+2)_3}
        &\leq\paren*{\frac{m+5}2-\frac{60}{(m+2)_3}+\frac{120}{(nm+2)_3}}\delta
        \leq\paren*{\frac{m}2+4}\delta,\label{eq:tvfista_new_error_2}
    \end{align}
    where we used $(nm+2)_3\geq(m+2)_3\geq60$ in the last inequality.
    Combining \cref{eq:tvfista_new_error_1,eq:tvfista_new_error_2} gives
    \begin{align}
        \Delta_{F_k}(x_k,u)
        &\leq\frac{30LD^2}{n^2m^2}+\paren*{\frac m2+4}\delta
    \end{align}
    for all $1\leq m\leq k$.
    By taking $m\coloneqq\min\set{k,\ceil{4\sqrt[3]{LD^2/(n^2\delta)}}}$, we obtain
    \begin{align}
        \Delta_{F_k}(x_k,u)
        &\leq\frac{30LD^2}{n^2}\max\set*{\frac1{k^2},\frac1{16}\paren*{\frac{n^2\delta}{LD^2}}^{2/3}}+\paren*{2\paren*{\frac{LD^2}{n^2\delta}}^{1/3}+5}\delta\\
        &\leq\frac{30LD^2}{n^2k^2}+4\sqrt[3]{\frac{LD^2\delta^2}{n^2}}+5\delta
    \end{align}
    for all $k\geq1$.
    Taking $u\in \mathcal X_{F_k}^*$ completes the proof.
\end{proof}

%% file: 93_appendix_tvapge.tex
\section{DETAILS FOR SECTION~\ref{sec:tvapge}}\label{sec:tvapge_details}
\subsection{Estimate of Temporal Variability}\label{sec:delta_estimation}
Fix an arbitrary reference point $\bar x\in\mathcal X$, and let $\delta_0>0$ be an initial estimate of $\delta$.
We assume $\delta_0\leq\delta$.
Such an estimate can be obtained by evaluating $|\Delta_{f_0}(x,\bar x)-\Delta_{f_{-1}}(x,\bar x)|/2$ at randomly chosen points, because each value is bounded above by $\delta$.
Here, $f_{-1}$ denotes the objective function immediately preceding $f_0$.
We assume \cref{asm:t_lipschitz} also holds for $k=-1$.
If $f_{-1}$ is unavailable in practice, we can initialize the estimate at $\delta_0=0$.
The initial estimate can also be zero when $f_{-1}$ is available but $\Delta_{f_0}(x,\bar x)=\Delta_{f_{-1}}(x,\bar x)$ at all sampled points.
Whenever $\delta_0=0$, we start with $\gamma_0=0$ and apply \cref{alg:fista} according to $(x_{k+1},z_{k+1},\gamma_{k+1})=\apg_1(f_k,x_k,z_k,\gamma_k,n,L)$.
We continue until a nonzero difference is detected and then switch to TV-APGe.
The estimate $\delta_k$ of $\delta$ at round $k$ is defined as
\begin{align}
    \delta_k\coloneqq\max\set*{\delta_0,\max_{1\leq i\leq k}\frac{\abs{\Delta_{f_i}(x_i,\bar x)-\Delta_{f_{i-1}}(x_{i},\bar x)}}2}\in[\delta_0,\delta].\label{eq:delta_estimation_v2}
\end{align}
The estimates $(\delta_k)_{k\geq0}$ form a nondecreasing sequence bounded above by $\delta$, so they converge.
We denote their limit by $\delta_\infty\in[\delta_0,\delta]$.
TV-APGe is summarized in \cref{alg:tvapge_full}.

In the experiments, we use $\bar x=0$.
We compute $\delta_0$ as the maximum of $|\Delta_{f_0}(x,0)-\Delta_{f_{-1}}(x,0)|/2$ over 10 randomly sampled points $x\in\mathcal X$ and the initial iterate $x_0$.
For the quartic experiment, we obtain $f_{-1}$ by setting $k=-1$ in the definition of $f_k$.
For the autoregression problems, we use $(f_{886},f_{887})$ in place of $(f_{-1},f_0)$.

\begin{algorithm}[t]
  \caption{TV-APGe}
  \label{alg:tvapge_full}
  \begin{algorithmic}[1]
    \Require $x_0,\bar x\in\mathcal X$, $n\geq1$, $L,D>0$, $\delta_{0}\in(0,\delta]$
    \State $z_0\gets x_0$
    \For{$k=0,1,\ldots$}
        \State $\gamma_{k}\gets\sqrt[3]{nLD^2/\delta_k}$
        \State $(x_{k+1},z_{k+1},\underline{\hspace{0.8em}})\gets\apg_1(f_k,x_k,z_k,\gamma_k,n,L)$
        \State $\delta_{k+1}\gets\max\{\delta_{k},\abs{\Delta_{f_{k+1}}(x_{k+1},\bar x)-\Delta_{f_{k}}(x_{k+1},\bar x)}/2\}$
    \EndFor
  \end{algorithmic}
\end{algorithm}

\subsection{Upper Bound on the Error of TV-APGe}
The theoretical analysis of TV-APGe is also based on the potential function.
Let $\tilde\Phi_k\coloneqq\tilde\Phi_{F_k,\gamma_k}(x_k,z_k,u)$ for arbitrarily fixed $u\in\mathcal X$.
\cref{prop:fista_potential_decrease} gives the following recursion on $\tilde\Phi_k$:
\begin{lemma}\label{lem:tvfista_general_recursion_v2}
    Suppose \cref{asm:problem} holds.
    Then, \cref{alg:tvapge_full} guarantees
    \begin{align}
        \frac1{\gamma_{k+1}^2}\paren*{\tilde\Phi_{k+1}-\frac{LD^2}2}
        &\leq\frac{\theta_k^2}{\gamma_k^2}\paren*{\tilde\Phi_k-\frac{LD^2}2}+2\delta\label{eq:tvfista_general_recursion_v2}
    \end{align}
    for all $k\geq0$, where $\theta_k\coloneqq\gamma_{k}/\phi^n(\gamma_k)$.
\end{lemma}

\begin{proof}
    From \cref{asm:t_lipschitz}, we have
    \begin{align}
        \frac{\tilde\Phi_{k+1}}{\gamma_{k+1}^2}
        &=\Delta_{F_{k+1}}(x_{k+1},u)+\frac L{2\gamma_{k+1}^2}\norm{z_{k+1}-u}^2\\
        &\leq\Delta_{F_{k}}(x_{k+1},u)+\frac L{2\gamma_{k+1}^2}\norm{z_{k+1}-u}^2+2\delta\\
        &=\frac1{(\phi^n(\gamma_k))^2}\paren*{(\phi^n(\gamma_k))^2\Delta_{F_k}(x_{k+1},u)+\frac L2\norm{z_{k+1}-u}^2}\\&\quad+\frac L2\paren*{\frac1{\gamma_{k+1}^2}-\frac1{(\phi^n(\gamma_k))^2}}\norm{z_{k+1}-u}^2+2\delta\\
        &\leq\theta_k^2\frac{\tilde\Phi_k}{\gamma_k^2}+\frac L2\paren*{\frac1{\gamma_{k+1}^2}-\frac1{(\phi^n(\gamma_k))^2}}\norm{z_{k+1}-u}^2+2\delta,
    \end{align}
    where the last inequality follows from \cref{prop:fista_potential_decrease}.
    By subtracting $LD^2/(2\gamma_{k+1}^2)$ from both sides, we obtain
    \begin{align}
        \frac1{\gamma_{k+1}^2}\paren*{\tilde\Phi_{k+1}-\frac{LD^2}2}
        &\leq\frac{\theta_k^2}{\gamma_k^2}\paren*{\tilde\Phi_k-\frac{LD^2}2}-\frac L2\paren*{\frac1{\gamma_{k+1}^2}-\frac1{(\phi^n(\gamma_k))^2}}\paren*{D^2-\norm{z_{k+1}-u}^2}+2\delta.
    \end{align}
    The second term on the right-hand side is nonpositive because $\norm{z_{k+1}-u}\leq D$ and $\gamma_{k+1}\leq\gamma_k\leq\phi^n(\gamma_k)$.
    Here, $\gamma_{k+1}\leq\gamma_k$ follows from $\delta_{k+1}\geq\delta_k$.
    Dropping this nonpositive term gives the desired inequality.
\end{proof}

Before presenting the main theorem, we give the following lemma.

\begin{lemma}\label{lem:upper_bound_of_product}
    Let $(a_k)_{k\geq0}$ be a non-decreasing sequence of positive numbers converging to $a_\infty\coloneqq\lim_{k\to\infty}a_k$.
    Then, we have
    \begin{align}
        \limsup_{k\to\infty}\sum_{i=0}^{k-1}\prod_{j=i+1}^{k-1}\paren*{1+a_j}^{-2}
        &\leq1+\frac1{2a_\infty}.
    \end{align}
\end{lemma}

\begin{proof}
    Let $\varepsilon\in(0,1)$ be an arbitrarily fixed small constant.
    Since $(a_k)_{k\geq0}$ converges to $a_\infty>0$, there exists $K\geq0$ such that $a_k/a_\infty\in[1-\varepsilon,1]$ for all $k\geq K$.
    Then, for all $k>K$, we have
    \begin{align}
        \sum_{i=0}^{k-1}\prod_{j=i+1}^{k-1}\paren*{1+a_j}^{-2}
        &=\sum_{i=0}^{K-1}\prod_{j=i+1}^{k-1}\paren*{1+a_j}^{-2}+\sum_{i=K}^{k-1}\prod_{j=i+1}^{k-1}\paren*{1+a_j}^{-2}.
    \end{align}
    The first term on the right-hand side converges to zero as $k\to\infty$ because
    \begin{align}
        \sum_{i=0}^{K-1}\prod_{j=i+1}^{k-1}\paren*{1+a_j}^{-2}
        &\leq\sum_{i=0}^{K-1}\prod_{j=K}^{k-1}\paren*{1+a_j}^{-2}
        \leq K\prod_{j=K}^{k-1}\paren*{1+a_K}^{-2}
        \to0.
    \end{align}
    The second term on the right-hand side can be bounded as
    \begin{align}
        \sum_{i=K}^{k-1}\prod_{j=i+1}^{k-1}\paren*{1+a_j}^{-2}
        &\leq\sum_{i=K}^{k-1}\prod_{j=i+1}^{k-1}\paren*{1+(1-\varepsilon)a_\infty}^{-2}\\
        &=\sum_{i=K}^{k-1}\paren*{1+(1-\varepsilon)a_\infty}^{-2(k-i-1)}
        =\sum_{i=0}^{k-K-1}\paren*{1+(1-\varepsilon)a_\infty}^{-2i}\\
        &\leq\paren*{1-(1+(1-\varepsilon)a_\infty)^{-2}}^{-1}\\
        &\leq1+\frac1{2(1-\varepsilon)a_\infty},
    \end{align}
    where the last inequality follows from $(1-(1+x)^{-2})^{-1}\leq1+1/(2x)$ for all $x>0$.
    Combining these inequalities guarantees that
    \begin{align}
        \limsup_{k\to\infty}\sum_{i=0}^{k-1}\prod_{j=i+1}^{k-1}\paren*{1+a_j}^{-2}
        &\leq1+\frac1{2(1-\varepsilon)a_\infty}
    \end{align}
    for all $\varepsilon\in(0,1)$, which implies the desired inequality.
\end{proof}

We are now ready to present the main theorem, which provides an upper bound on the asymptotic function value gap of \cref{alg:tvapge_full}.
The following theorem gives the full version of \cref{thm:tvfista_final_error_short} stated in the main text.

\begin{theorem}\label{thm:tvfista_final_error}
    Suppose \cref{asm:problem} holds.
    Then, \cref{alg:tvapge_full} guarantees that
    \begin{align}
        \Delta_k
        &\leq\frac12\sqrt[3]{\frac{LD^2\delta_k^2}{n^2}}+\Delta_0\prod_{i=0}^{k-1}\paren*{1+\frac 1{2}\sqrt[3]{\frac{n^2\delta_i}{LD^2}}}^{-2}+2\delta\sum_{i=0}^{k-1}\prod_{j=i+1}^{k-1}\paren*{1+\frac12\sqrt[3]{\frac{n^2\delta_j}{LD^2}}}^{-2}\label{eq:tvfista_final_error}
    \end{align}
    for all $k\geq0$, and therefore
    \begin{align}
        \limsup_{k\to\infty}\Delta_k
        &\leq\paren*{\frac12+2\sqrt[3]{\frac{\delta}{\delta_\infty}}}\sqrt[3]{\frac{LD^2\delta^2}{n^2}}+2\delta.\label{eq:tvfista_final_error_asymptotic}
    \end{align}
\end{theorem}

\begin{proof}
    Solving the recursion~\cref{eq:tvfista_general_recursion_v2} gives
    \begin{align}
        \frac1{\gamma_k^2}\paren*{\tilde\Phi_k-\frac{LD^2}2}
        &\leq\frac1{\gamma_0^2}\paren*{\tilde\Phi_0-\frac{LD^2}2}\prod_{i=0}^{k-1}\theta_i^2+2\delta\sum_{i=0}^{k-1}\prod_{j=i+1}^{k-1}\theta_j^2.
    \end{align}
    Since $\tilde\Phi_k\geq\gamma_k^2\Delta_{F_k}(x_k,u)$ and $\tilde\Phi_0\leq\gamma_0^2\Delta_0+LD^2/2$, we have
    \begin{align}
        \Delta_{F_k}(x_k,u)
        &\leq\frac{LD^2}{2\gamma_k^2}+\Delta_0\prod_{i=0}^{k-1}\theta_i^2+2\delta\sum_{i=0}^{k-1}\prod_{j=i+1}^{k-1}\theta_j^2.
    \end{align}
    Since $\phi(\gamma)\geq\gamma+1/2$ for all $\gamma\geq0$, we have 
    \begin{align}
        \theta_k
        &=\frac{\gamma_k}{\phi^n(\gamma_k)}
        \leq\paren*{1+\frac n{2\gamma_k}}^{-1}.
    \end{align}
    Using this inequality, substituting $\gamma_k=\sqrt[3]{nLD^2/\delta_k}$, and taking $u\in\mathcal X_{F_k}^*$ gives \cref{eq:tvfista_final_error}.

    Since $\delta_k\geq\delta_0>0$ for all $k\geq0$, the product in the second term on the right-hand side of \cref{eq:tvfista_final_error} satisfies
    \begin{align}
        \lim_{k\to\infty}\prod_{i=0}^{k-1}\paren*{1+\frac 1{2}\sqrt[3]{\frac{n^2\delta_i}{LD^2}}}^{-2}
        &=0.
    \end{align}
    For the third term in the right-hand side of \cref{eq:tvfista_final_error}, we have
    \begin{align}
        \limsup_{k\to\infty}\sum_{i=0}^{k-1}\prod_{j=i+1}^{k-1}\paren*{1+\frac12\sqrt[3]{\frac{n^2\delta_j}{LD^2}}}^{-2}
        &\leq1+\sqrt[3]{\frac{LD^2}{n^2\delta_\infty}}
    \end{align}
    from \cref{lem:upper_bound_of_product}.
    Combining these inequalities and using $\delta_\infty\leq\delta$ gives \cref{eq:tvfista_final_error_asymptotic}.
\end{proof}

When we have access to $\delta$, we can set $\delta_0=\delta$ in \cref{alg:tvapge_full} and guarantee a non-asymptotic bound on $\Delta_k$ as follows.

\begin{corollary}\label{cor:tvfista_final_error_v3}
    Suppose \cref{asm:problem} holds.
    Then, for all $k\geq0$, \cref{alg:tvapge_full} with $\delta_0=\delta$ guarantees that
    \begin{align}
        \Delta_k
        &\leq\frac52\sqrt[3]{\frac{LD^2\delta^2}{n^2}}+2\delta+\paren*{1+\frac 1{2}\sqrt[3]{\frac{n^2\delta}{LD^2}}}^{-2k}\Delta_0.
    \end{align}
\end{corollary}

\begin{proof}
    Since $\delta_0=\delta$ implies $\delta_k=\delta$ for all $k\geq0$, we have
    \begin{align}
        \Delta_k
        &\leq\frac12\sqrt[3]{\frac{LD^2\delta^2}{n^2}}+\Delta_0\prod_{i=0}^{k-1}\paren*{1+\frac 1{2}\sqrt[3]{\frac{n^2\delta}{LD^2}}}^{-2}+2\delta\sum_{i=0}^{k-1}\prod_{j=i+1}^{k-1}\paren*{1+\frac12\sqrt[3]{\frac{n^2\delta}{LD^2}}}^{-2}\\
        &\leq\frac12\sqrt[3]{\frac{LD^2\delta^2}{n^2}}+\paren*{1+\frac 1{2}\sqrt[3]{\frac{n^2\delta}{LD^2}}}^{-2k}\Delta_0+2\delta\sum_{i=0}^{k-1}\paren*{1+\frac12\sqrt[3]{\frac{n^2\delta}{LD^2}}}^{-2(k-i-1)}
    \end{align}
    from \cref{eq:tvfista_final_error}.
    Using the same argument as in the proof of \cref{lem:upper_bound_of_product}, we bound the last term as follows.
    \begin{align}
        2\delta\sum_{i=0}^{k-1}\paren*{1+\frac12\sqrt[3]{\frac{n^2\delta}{LD^2}}}^{-2(k-i-1)}
        &\leq2\delta\paren*{1-\paren*{1+\frac12\sqrt[3]{\frac{n^2\delta}{LD^2}}}^{-2}}^{-1}
        \leq2\delta\paren*{1+\sqrt[3]{\frac{LD^2}{n^2\delta}}},
    \end{align}
    where the last inequality follows from $(1-(1+x)^{-2})^{-1}\leq1+1/(2x)$ for all $x>0$.
    Combining these inequalities gives the desired inequality.
\end{proof}

%% file: 94_appendix_numerical.tex
\section{EXPERIMENTAL DETAILS AND ADDITIONAL RESULTS}\label{sec:exp_details}
\subsection{Details Common to All Experiments}\label{sec:common_details}
We used an Apple M4 (10 cores, 32\,GiB RAM), macOS 26.6 (arm64), Python 3.13.5, NumPy 2.5.3, SciPy 1.18.1, and Apple Accelerate BLAS.
We report the mean function value gap over 30 independent runs, together with pointwise 95\% confidence intervals based on Student's $t$ distribution.
For the quartic experiment, we vary the random seed used to generate the problem instance.
For the autoregression experiment, we vary the random seed used to select the time series.
All other random seeds are fixed.

\subsection{Details of the Quartic Function}\label{sec:quartic_details}
The problem considered in \cref{sec:quartic} is as follows:
\begin{align}
    \min_{x\in\mathcal X}f_k(x)\coloneqq \frac14\norm{x-y_k}_{A_k}^4,\quad\mathcal X\coloneqq[-1,1]^d.
\end{align}
We set $d=100$, $r=90$, and $t_k=kh$.
Let $\lambda_i\coloneqq10^{-(i-1)/(r-1)}$, and let $S\in\R^{d\times d}$ be a randomly generated skew-symmetric matrix normalized so that $\|S\|_2=1$.
We define $A_k\coloneqq\e^{t_kS}\diag(\lambda_1,\ldots,\lambda_r,0,\ldots,0)\e^{-t_kS}$.
The components of $y_k$ are given by $y_{k,i}\coloneqq\sin(\omega_i t_k)$, where $\omega_i\sim\operatorname{Unif}(0,1)$.
We use sampling periods $h\in\{0.1,0.05,0.02,0.01,0.005,0.002,0.001\}$.
For each value of $h$, the methods are run over all $k$ such that $t_k\in[0,10]$.
This problem satisfies \cref{asm:problem} with $D=2\sqrt d$, $L=3D^2$, and $\delta=O(h)$.
We set $x_0=(1,\ldots,1)^\top$.
The parameter $v$ of ORPG is set to $v=h\norm{\omega}$, since $\norm*{y_{k+1}-y_k}\leq h\norm{\omega}$.

\subsection{Details of the Autoregression Problems}\label{sec:autoregression_details}
In \cref{sec:autoregression}, we use three time-series datasets from the Monash Time Series Forecasting Archive~\citep{godahewa2021monash}: Electricity Hourly (Electricity), Traffic Hourly (Traffic), and Solar 10 Minutes (Solar).
Electricity and Traffic are based on real-world measurements, whereas Solar contains simulated solar power data.
The Electricity data originate from the \texttt{ElectricityLoadDiagrams20112014} dataset in the UCI Machine Learning Repository~\citep{trindade2015electricity}, the Traffic data from the California Department of Transportation Performance Measurement System (PeMS)~\citep{caltrans_pems}, and the Solar data from the Solar Power Data for Integration Studies provided by the National Renewable Energy Laboratory~\citep{nrel_solar}.
Specifically, we use the files \texttt{electricity\_hourly\_dataset.tsf}~\citep{godahewa2020electricityhourly}, \texttt{traffic\_hourly\_dataset.tsf}~\citep{godahewa2020traffichourly}, and \texttt{solar\_10\_minutes\_dataset.tsf}~\citep{godahewa2020solar10min}.

We use the first $17544$ rounds from each dataset, and we set the calibration period to the first $720$ rounds.
Time series with a population standard deviation of at most $10^{-8}$ during the calibration period are excluded, and $128$ time series are randomly sampled without replacement from the remaining series.
The time series are standardized using the mean and population standard deviation during the calibration period, ensuring that future data are not used in the calculation of these statistics.

For each $\mathcal D$, let $x^{*,\mathcal D}$ be an approximate minimizer of the calibration objective.
This objective is obtained from \cref{eq:autoregression} by replacing $\mathcal X$ with $\R^{1+168}$ and $\mathcal D_k$ with the dataset from the calibration period.
We initialize the methods at $x_{887}=0$ and set $R=\max\set{1,2\max_{\mathcal D}\norm{x^{*,\mathcal D}}}$, where $\max_{\mathcal D}$ is the maximum over the three datasets.
The minimizer is computed by the L-BFGS-B implementation in SciPy~\citep{virtanen2020scipy}. 
Since $\ell$ is 1-smooth, we have
\begin{align}
    \norm{\nabla^2f_k(x)}
    &\leq\frac1{|\mathcal D_k|}\sum_{(a,b)\in\mathcal D_k}\norm{a}^2
    \leq\max_{k=887,\ldots,17543}\frac1{|\mathcal D_k|}\sum_{(a,b)\in\mathcal D_k}\norm{a}^2.
\end{align}

For each dataset, we compute the rightmost quantity in the preceding inequality in advance and use it as $L$.
Although this uses future information, it ensures that differences in the accuracy of estimates of $L$ do not influence the comparison.
The same value of $L$ can still affect the methods differently because their step-size rules use it in different ways.
For calibration and initialization, we also define $f_k$ and $\mathcal D_k$ for $k=335,\ldots,886$ using the formulas in \cref{sec:autoregression}.
The parameter $v$ for ORPG is set to $v\coloneqq\max_{\mathcal D,k=335,\ldots,718}\norm*{x_{k+1}^{*,\mathcal D}-x_{k}^{*,\mathcal D}}$, where $x_k^{*,\mathcal D}$ is the approximate minimizer of the objective function~\cref{eq:autoregression}.
We compute the calibration reference solutions used to estimate $v$ and the reference values used to plot optimality gaps using SLSQP in SciPy~\citep{virtanen2020scipy}.
When necessary, the solutions are refined by PG iterations with step size $1/L_k^{\mathrm{loc}}$, where
\begin{align}
    L_k^{\mathrm{loc}}\coloneqq
    \lambda_{\max}\paren*{\frac{1}{|\mathcal D_k|}\sum_{(a,b)\in\mathcal D_k}aa^\top}.
\end{align}
We require the gradient mapping norm to be at most $10^{-6}$.
The mapping is evaluated with step size $1/L_k^{\mathrm{loc}}$ for the calibration references and $1/L$ for the references used to plot optimality gaps.
Although this approximation may affect the absolute values of the plotted gaps, it does not affect the ordering of the methods, since the same approximation of the optimal value is used for all methods and all plotted values are positive.

\subsection{Additional Numerical Results}\label{sec:additional_numerical_results}
We present additional numerical results comparing TV-APG and TV-APGe with two methods that apply $\apg_1$ directly and two methods proposed in the context of OCO.
In addition to the four methods in \cref{sec:numerical}, we implement TV-APGr, TV-APGn, SWORD, and SWORD++~\citep{zhao2024adaptivity}.

TV-APGr restarts $\apg_1$ at each round, using $(x_{k+1},z_{k+1},\underline{\hspace{0.8em}})\gets\apg_1(f_k,x_k,x_k,0,n,L)$.
TV-APGn is a naive extension of $\apg_1$ to the time-varying setting.
It continues the $\apg_1$ iterations across rounds without restarting, using $(x_{k+1},z_{k+1},\gamma_{k+1})=\apg_1(f_k,x_k,z_k,\gamma_k,n,L)$ with $\gamma_0=0$.

\subsubsection{Adaptation of SWORD and SWORD++}
We adapt SWORD and SWORD++ of \citet{zhao2024adaptivity} to a budget of $10$ gradient queries per round.

\paragraph{SWORD.}
We replace the original step size candidates with
\begin{align}
    \eta_1=\frac1{4L},\quad\eta_8=\min\set*{\frac{D}{G\sqrt{8K}},\frac1{4L}},\quad\eta_i=\eta_1\paren*{\frac{\eta_8}{\eta_1}}^{(i-1)/7},~i=2,\ldots,7,
\end{align}
where $K$ is the number of rounds with updates ($10/h$ for the quartic problem and $16656$ for autoregression) and $G$ is a constant such that $\norm{\nabla f_k(x)}\leq G$ for all $x\in\mathcal X$ and $k$.
We set $G=D^3$ for the quartic function and $G=\sqrt L$ for the autoregression problems.
Indeed, for the quartic function, we have
\begin{align}
    \norm{\nabla f_k(x)}
    &=\norm{A_k(x-y_k)}\norm{x-y_k}_{A_k}^2
    \leq\norm{A_k}^2\norm{x-y_k}^3
    \leq D^3,
\end{align}
and for the autoregression problems, we have 
\begin{align}
    \norm{\nabla f_k(x)}
    &=\norm{\frac1{|\mathcal D_k|}\sum_{(a,b)\in\mathcal D_k}\ell'(\inner{a,x}-b)a}
    \leq\frac1{|\mathcal D_k|}\sum_{(a,b)\in\mathcal D_k}\norm{a}
    \leq \sqrt L,
\end{align}
where the first inequality uses the fact that $\ell$ is 1-Lipschitz continuous and the last inequality uses the Cauchy--Schwarz inequality and the setting of $L$ in this experiment.
The rest of the algorithm is unchanged.

\paragraph{SWORD++.}
Instead of one update per physical round, we perform ten consecutive updates on the same loss $f_k$, each using one gradient shared by all experts. 
Let $x_k^i$ denote the $i$th update in round $k$ for $i=0,\ldots,9$, with $x_k^0=x_k$.
We retain the step size candidates of Equation~(26) in the original paper, using the virtual horizon
$10K$ instead of $K$.
The rest of the algorithm is unchanged.

\paragraph{Meta Learning Rate.}
We use the adaptive learning rate version in \citet[Appendix B]{zhao2024adaptivity}.
To compute $x_k^i$, we set the learning rate to
\begin{align}
    \varepsilon_{k,i}=
    \min\left\{\frac{1}{cD^2L},
    \sqrt{\frac{\log N}{D^2\sum_{(\ell,j)<_{\mathrm{lex}}(k,i)}\|\nabla f_{\ell}(x_\ell^j)-M_{\ell,j}\|^2}}\right\},
\end{align}
where $N$ is the number of experts, $c=4$ for SWORD and $c=8$ for SWORD++, and $<_{\mathrm{lex}}$ denotes the lexicographic order.
For SWORD, we restrict $i=j=0$, set $x_\ell^0=x_\ell$, and use $M_{\ell,0}=\nabla f_{\ell-1}(\bar x_\ell)$ for $\ell\geq1$, where $\bar x_\ell$ is the instrumental point defined in
\citet[Equation~(15)]{zhao2024adaptivity}.
For SWORD++, we set $M_{\ell,j}=\nabla f_\ell(x_\ell^{j-1})$ for $j\geq1$ and $M_{\ell,0}=\nabla f_{\ell-1}(x_{\ell-1}^9)$ for $\ell\geq1$.
In both cases, $M_{0,0}=0$.
We use $1/(cD^2L)$ when the sum in the denominator is zero.

\subsubsection{Results}
\cref{fig:quartic_appendix} shows the results for the quartic function.
\Cref{fig:quartic_001_appendix} plots the optimality gaps every 10 rounds when $h=0.01$, and \cref{fig:quartic_h_appendix} plots the gap at the last round for each value of $h$, where $h$ is the sampling period and $\delta=O(h)$ holds.
TV-APGr performed better than TV-PG thanks to the acceleration, but it performed much worse than TV-APGe and TV-APG.
TV-APGn performed best in the early rounds, but its performance deteriorated in the later rounds.
The performance of SWORD and SWORD++ was worse than that of TV-PG.

\begin{figure*}[tb]
    \centering
    \includegraphics[width=\columnwidth]{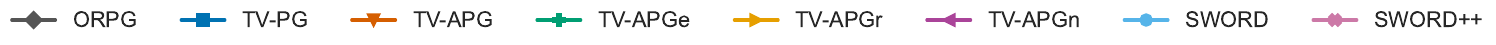}\\
    \begin{minipage}[b]{0.49\textwidth}
        \centering
        \includegraphics[width=\columnwidth]{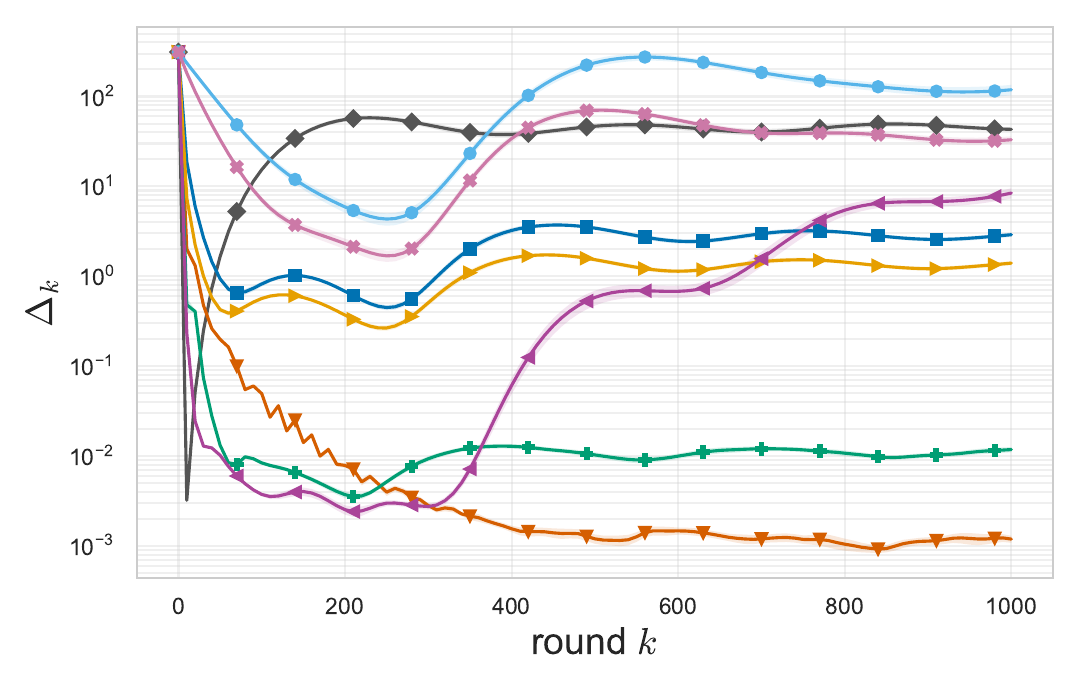}
        \subcaption{Gaps every 10 rounds when $h=0.01$.}
        \label{fig:quartic_001_appendix}
    \end{minipage}
    \begin{minipage}[b]{0.49\textwidth}
        \centering
        \includegraphics[width=\columnwidth]{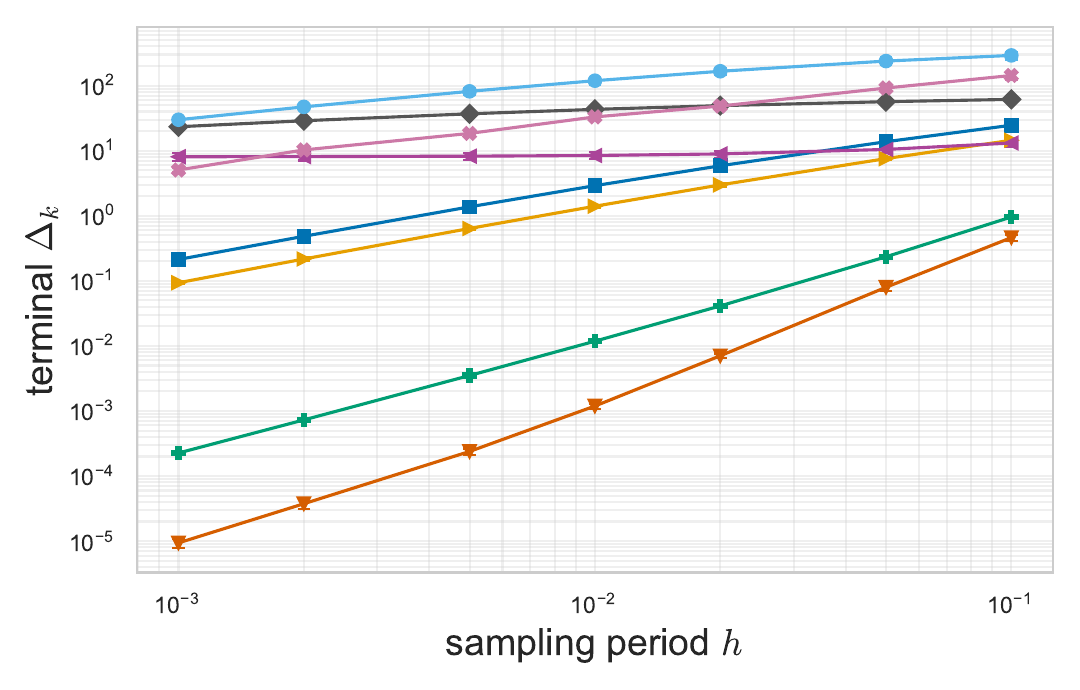}
        \subcaption{Terminal gaps for each $h$.}
        \label{fig:quartic_h_appendix}
    \end{minipage}
    \caption{Quartic Function. \cref{fig:quartic_001_appendix} plots the optimality gaps $\Delta_k$ when $h=0.01$ every 10 rounds. \cref{fig:quartic_h_appendix} plots $\Delta_k$ at the final round for each sampling period $h$. The confidence intervals are shown as shaded bands and error bars in \cref{fig:quartic_001_appendix,fig:quartic_h_appendix}, respectively, but they are too small to be visible.}
    \label{fig:quartic_appendix}
\end{figure*}

\cref{fig:autoregression_appendix} shows the results for the autoregression problems in the same format as \cref{sec:autoregression}.
TV-APGr performed better than TV-PG, but it performed worse than TV-APGe.
TV-APGn performed comparably to TV-APGe, matching its strong empirical performance, in \cref{fig:electricity_appendix}.
However, in the other two datasets, its performance deteriorated in the later rounds similarly to the quartic function experiment.
The performance of SWORD and SWORD++ was worse than that of TV-PG.

These results illustrate a limitation of applying $\apg_1$ directly to TVCO problems.
As discussed in \cref{sec:apg}, the difficulty arises because the parameter $\gamma$ depends on the number of rounds.
The results also show that SWORD and SWORD++ do not outperform TV-PG in the TVCO setting considered here, where the objective function changes gradually over time. 
This may partly reflect the difference in the settings for which the methods are designed: SWORD and SWORD++ provide theoretical guarantees under substantially more general online environments, including cases where the objective function may change abruptly across rounds, whereas TV-PG specifically exploits the gradual temporal variation assumed in our TVCO setting.

\begin{figure*}[tb]
    \centering
    \includegraphics[width=\columnwidth]{figures/legend_eight_methods.pdf}\\
    \begin{minipage}[b]{0.49\textwidth}
        \centering
        \includegraphics[width=\columnwidth]{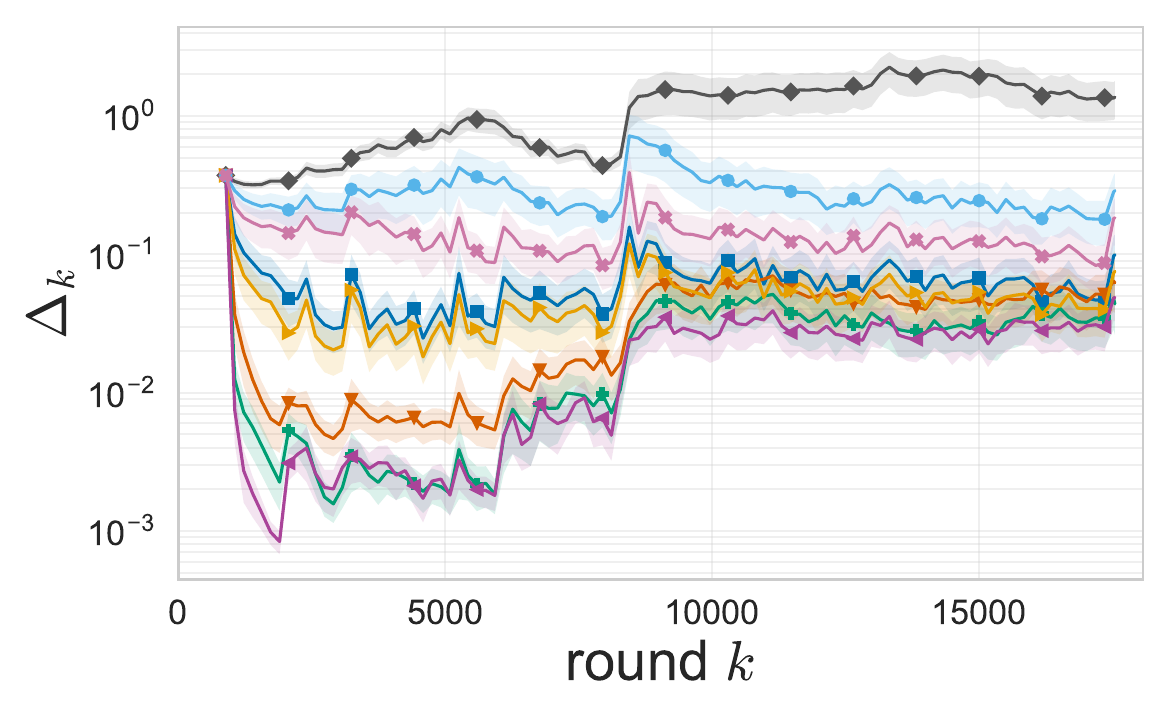}
        \subcaption{Electricity.}
        \label{fig:electricity_appendix}
    \end{minipage}
    \begin{minipage}[b]{0.49\textwidth}
        \centering
        \includegraphics[width=\columnwidth]{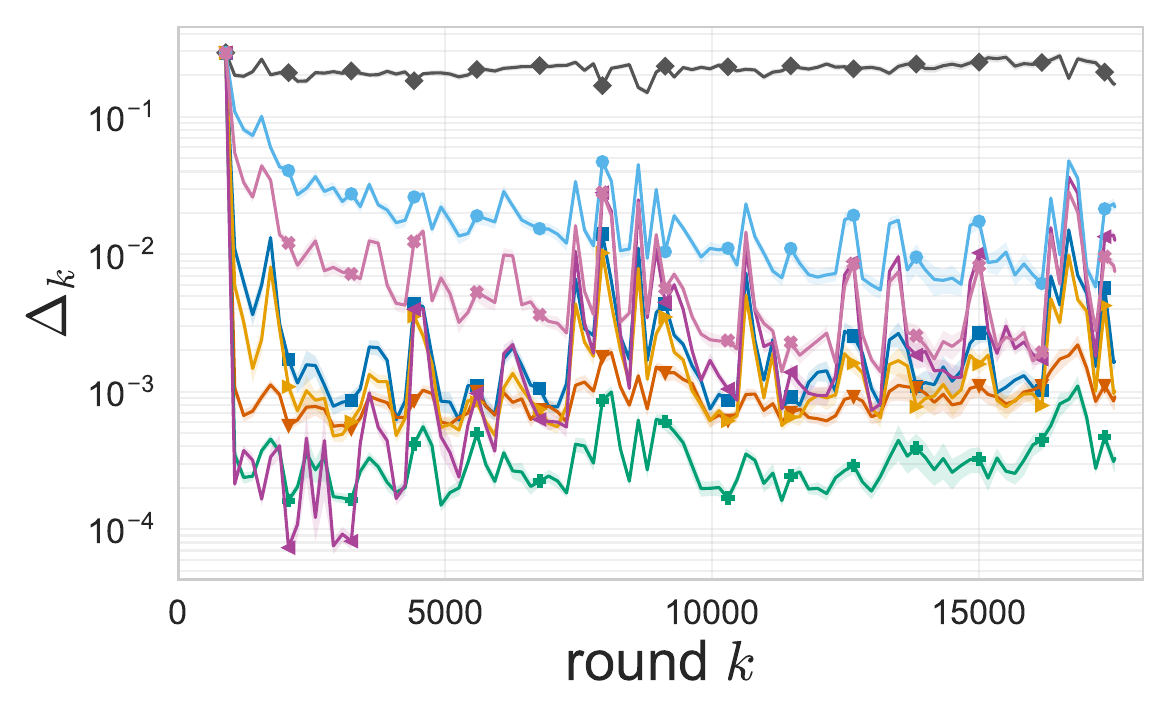}
        \subcaption{Traffic.}
        \label{fig:traffic_appendix}
    \end{minipage}\\
    \begin{minipage}[b]{0.49\textwidth}
        \centering
        \includegraphics[width=\columnwidth]{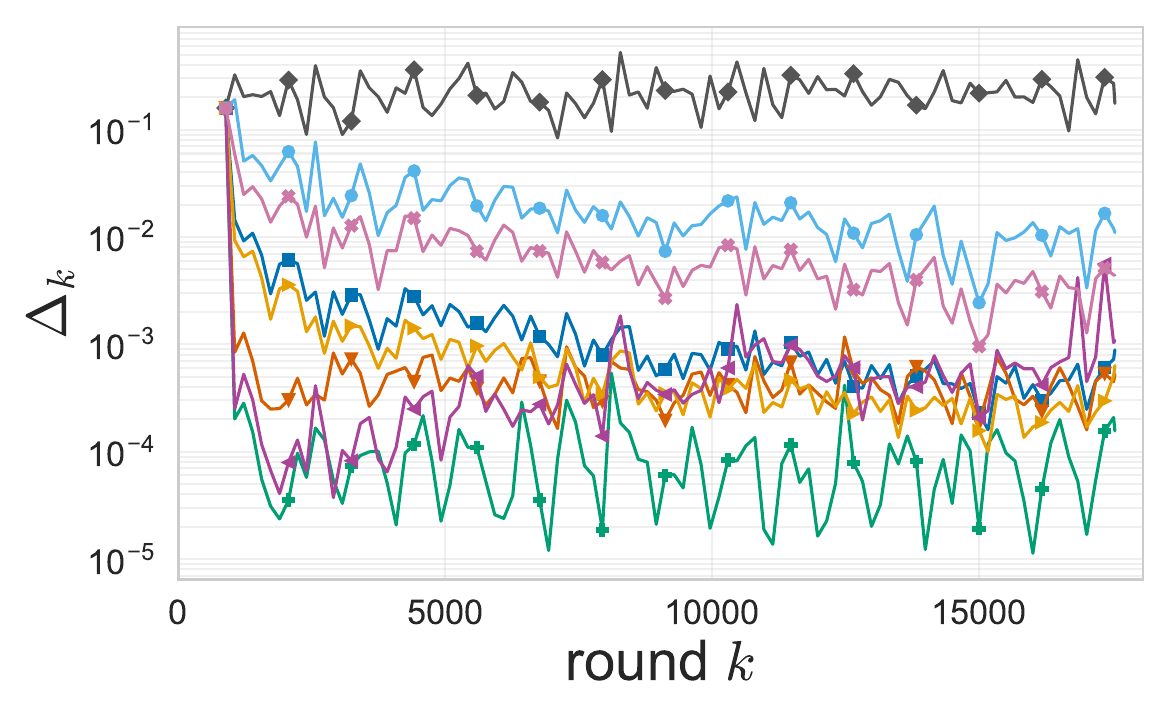}
        \subcaption{Solar.}
        \label{fig:solar_appendix}
    \end{minipage}
    \caption{Autoregression. The optimality gaps $\Delta_k$ are plotted every $168$ rounds and at the final round. The confidence intervals are shown as shaded bands.}
    \label{fig:autoregression_appendix}
\end{figure*}

\begin{table}[tb]
  \centering
  \caption{Computational Time per Round. Values are in milliseconds (mean $\pm$ sample standard deviation of five repeat means). Quartic uses $h=0.01$.}
  \label{tab:computational_time}
  \setlength{\tabcolsep}{3pt}
  \begin{tabular}{c|c|c|c|c}
    \toprule
    \textbf{Method} & \textbf{Quartic} & \textbf{Electricity} & \textbf{Traffic} & \textbf{Solar} \\
    \midrule
    ORPG & $0.441 \pm 0.027$ & $9.509 \pm 0.632$ & $9.445 \pm 0.669$ & $9.331 \pm 0.572$ \\
    TV-PG & $0.443 \pm 0.027$ & $9.536 \pm 0.672$ & $9.469 \pm 0.707$ & $9.343 \pm 0.575$ \\
    TV-APG & $0.478 \pm 0.029$ & $9.625 \pm 0.738$ & $9.535 \pm 0.735$ & $9.412 \pm 0.617$ \\
    TV-APGe & $0.524 \pm 0.031$ & $11.693 \pm 0.972$ & $11.616 \pm 1.019$ & $11.417 \pm 0.847$ \\
    TV-APGr & $0.465 \pm 0.029$ & $9.549 \pm 0.646$ & $9.500 \pm 0.697$ & $9.382 \pm 0.638$ \\
    TV-APGn & $0.465 \pm 0.028$ & $9.584 \pm 0.694$ & $9.488 \pm 0.699$ & $9.375 \pm 0.625$ \\
    SWORD & $0.453 \pm 0.023$ & $9.570 \pm 0.652$ & $9.504 \pm 0.721$ & $9.372 \pm 0.593$ \\
    SWORD++ & $0.763 \pm 0.031$ & $9.986 \pm 0.886$ & $9.958 \pm 0.903$ & $9.811 \pm 0.802$ \\
    \bottomrule
  \end{tabular}
\end{table}

\subsection{Computational Time}\label{sec:computational_time}
We report the average computational time per round for each method in \cref{tab:computational_time}. 
For each problem, we timed a single fixed instance over five repeats. 
In each repeat, all methods were reinitialized, run for 24 warm-up rounds, and then timed over the subsequent 240 rounds, with the warmed-up states carried forward. 
The reported values are the means $\pm$ sample standard deviations of the five repeat means, where each repeat mean is averaged over the 240 timed rounds.
The timings exclude construction of the autoregression design matrices and response vectors.

All methods have similar computational times, except for TV-APGe and SWORD++.
TV-APGe takes longer due to the additional function evaluations required for its updates.
SWORD++ takes longer for the quartic function because it uses five experts, while the number of experts is one or two for the autoregression problems.
The computational time required for algebraic operations, gradient evaluations, and function evaluations may vary depending on the problem setting and implementation.
Although the measured per-round computational times differ across methods, we fix the number of gradient queries per round to 10 for all methods in our experiments.